\documentclass[twoside,11pt,reqno]{amsart}
\usepackage{amsfonts, amsbsy, amsmath, amsthm, amssymb, latexsym, verbatim, enumerate}

\usepackage[english]{babel}
\usepackage{float}

\newtheorem{propo}{Proposition}[section]
\newtheorem{defi}[propo]{Definition}

\newtheorem{lemma}[propo]{Lemma}
\newtheorem{corol}[propo]{Corollary}

\newtheorem{theo}[propo]{Theorem}

\newtheorem{rem}[propo]{Remark}

\theoremstyle{definition}

\newtheorem{remk}[propo]{Remark}

\newcommand{\ld}{,\ldots ,}

\newcommand{\ra}{ \rightarrow }

\newcommand{\diag}{\mathop{\rm diag}\nolimits}

\newcommand{\Om}{\Omega}

\newcommand{\ZZ}{\mathop{\Bbb Z}\nolimits}

\newcommand{\al}{\alpha}
\newcommand{\be}{\beta}

\newcommand{\ep}{\varepsilon}

\newcommand{\lam}{\lambda }

\newcommand{\up}{^{-1}}

\def\11{$(1)$}
\def\22{$(2)$}
\def\33{$(3)$}

\def\d12{{_{12}}}

\def\ei{{eigenvalue }}
\def\eis{{eigenvalues }}

\def\f{{following }}

\def\ii{{if and only if }}

\def\ir{{irreducible }}

\def\irr{{irreducible representation }}

\def\itf{{It follows that }}

\def\mult{{multiplicity }}

\def\rep{{representation }}
\def\reps{{representations }}

\def\St{{Suppose that }}

\def\SL{{\rm SL}}
\def\SO{{\rm SO}}
\def\GL{{\rm GL}}
\def\Sp{{\rm Sp}}

\newcommand{\el}{\end{lemma}}
\newcommand{\om}{\omega }
\newcommand{\med}{\medskip }

\def\ag{algebraic group }
\def\hw{highest weight }
\newcommand{\bp}{\begin{proof}}
\newcommand{\enp}{\end{proof}}
\newcommand{\bl}{\begin{lemma}\label}

\begin{document}

\title[Almost cyclic regular elements]{ Almost cyclic regular elements in irreducible representations of simple algebraic groups}
\author{D.M. Testerman and A.E. Zalesski}
\address{Donna Testerman, Ecole Polytechnique F\'ed\'eral de Lausanne,
\newline
\hskip1cm Institute of Mathematics, Station 8, CH-1015 Lausanne, Switzerland\newline
e-mail: donna.testerman@epfl.ch}

\address{Alexandre E. Zalesski, Department of Physics, Mathematics and Informatics,\newline
Academy of Sciences of Belarus, 66 Prospekt Nezalejnasti, Minsk, 220000, Belarus\newline
e-mail: alexandre.zalesski@gmail.com}
\keywords{simple algebraic groups, representations, eigenvalue multiplicities, regular elements}
\subjclass{20G05}
\thanks{The first author acknowledges the support of the Swiss National Science Foundation, research grant number $200020\textunderscore175571$. In addition, this work was initiated while the second author visited the EPFL; he gratefully acknowledges this institute's support.}
\maketitle

{\it Abstract} Let $G$ be a simple linear algebraic group  defined over an algebraically closed field of characteristic $p\geq 0$ and
let $\phi$ be a $p$-restricted irreducible representation of $G$. Let  $T$ be  a maximal torus of $G$ and $s\in T$.
 We say that $s$ is \emph{strongly regular} if $\al(s)\neq \be(s)$ for all distinct $T$-roots $\alpha$ and $\beta$ of $G$.
 Our main result states that if all but one of the \eis of $\phi(s)$ are of multiplicity 1 then, with a few specified exceptions, $s$ is
 strongly regular. This  can be viewed as an extension of
 our earlier result saying that under the same hypotheses, $s$ must be regular and all non-zero weights of $\phi$ are of \mult 1.

\def\ss{semisimple }

\section{Introduction}

Fix an algebraically closed field $F$ of characteristic $p\geq 0$, let $G$ be a simple linear \ag defined over $F$ (hereafter referred to as an
\emph{algebraic group}).   Let $\phi:G\ra \GL(V)$
be a rational \ir representation (so $V$ is a finite-dimensional
$FG$-module).
Recall that $g\in G$ is said to be \emph{regular} if $\dim C_G(g)={\rm rank}(G)$; so in particular, for a semisimple element $s\in G$,
$s$ is regular if and only if $C_G(s)^\circ$ is a maximal torus of $G$. A starting point of our project is the following observation:

\med
\noindent $(*)$ Let $s\in G$ be semisimple. If $\phi(s)$ is a regular element in $ \GL(V)$, then $s$ is regular in $G$.

\med
Note that the assumption on $\phi(s)$ in $(*)$ is equivalent to saying that all \ei multiplicities of $\phi(s)$
are equal to 1.

Let $T$ be a maximal torus of $G$. As $T$ meets every conjugacy class of \ss elements of $G$, for our considerations, we can assume that $s\in T$.
Recall that  the roots of $G$ with respect to $T$ are the non-trivial \ir constituents of the restriction
${\rm Ad}|_T$, where ${\rm Ad}$ is the adjoint \rep of $G$. Then $s$ is regular \ii $\al(s)\neq 1$, for every root $\al$ of the
root system of $G$, see \cite[Ch. III,  \S 1, Corollary 1.7]{Spr}. We generalize this in the following
\begin{defi} Let $s\in T$. We say that $s$ is {\it strongly regular} if
  $\al(s)\neq \be(s)$ whenever $\al\neq \be$ are $T$-roots of $G$.
\end{defi}

A strongly regular element is regular as $\al(s)=1$ implies $(-\al)(s)=1$; however,
there are many elements in $T$ that are regular but not strongly regular. Our first main result enforces $(*)$ as follows:

\begin{theo}
\label{th1}
 Let $G$ be a simply connected simple  algebraic group   defined over $F$. Let $s\in G$ be a semisimple element and
let $\phi $ be an \ir representation of $G$ with $p$-restricted \hw $\om\neq 0$.
Suppose that all of the \eis of $\phi(s)$ are of \mult $1$.
Then
one of
the following holds:
\begin{enumerate}[]

\item{\rm{(1)}} the element $s$ is strongly regular;

\item{\rm {(2)}} $G=A_n$, $n\geq 1$, with $p\neq 2$ if $n=1$, and $\om\in\{\om_1\ld \om_{n}\}$;
\item{\rm{(3)}}  $G=B_n$, $n\geq 3$, and  $\om\in\{\om_1,\om_n\}$;
  \item{\rm{(4)}} $G=C_n$, $n\geq  2$, $p= 2$,  and
 $\om\in\{\om_1,\om_n\}$;
\item{\rm{(5)}}  $G=C_n$, $n\geq 2$, $p\neq 2$, and either\begin{enumerate}[]
  \item{\rm{(a)}} $\om=\om_1$, or
  \item{\rm{(b)}} $n\in\{2,3\}$ and $\om=\om_n$, or
  \item{\rm{(c)}} $n\geq 3$, $p=3$, and $\om\in\{\om_{n-1},\om_n\}$;
    \end{enumerate}

\item{\rm{(6)}}  $G=D_n$, $n\geq 4,$ and $\om\in\{\om_1,\om_{n-1},\om_n\}$;
\item{\rm{(7)}} $G=E_6$ and  $\om\in\{\om_1,\om_6\}$;
\item{\rm{(8)}} $G=E_7$ and  $\om=\om_7$;
\item{\rm{(9)}}  $G=F_4$, $p=3$, and $\om=\om_4$;
\item{\rm {(10)}}  $G=G_2$ and  either $\om=\om_1$, or $p= 3$ and $\om=\om_2$.

 \end{enumerate}
\end{theo}

(We refer the reader to the end of this section for an explanation of the notation used in the statement.
For the reader's convenience the highest weights given in items (2)--(10) are collected in Table~\ref{tab:results}, at this end of this manuscript.)

The above result, as well as Theorems~\ref{re1}, \ref{rr4},  \ref{td4} below, can be viewed as an identification result,
in the following sense:
the
existence of an element whose spectrum satisfies a specified criteria on a given representation either determines the
representation or provides structural information about  the element. We expect that strongly regular
elements can play a certain role in the general theory of algebraic groups and finite groups of Lie type.
We mention a paper by  Seitz \cite[Lemma 2]{S93}, where he proves that $\al|_{T(q)}\neq \be |_{T(q)}$ for the maximal tori $T(q)$
 of finite quasisimple groups of Lie type $G(q)$, with $q>5$. This allows him to obtain a certain upper bound for
 the homogeneous components of $T(q)$ on \ir $G(q)$-modules. It was mentioned in \cite[p.179]{GS} that an analogous bound holds for the
 $s$-eigenspace dimensions, $s\in G$ semisimple,  provided $s$ is strongly regular.

 In fact,  we are able to obtain nearly the same conclusion as in Theorem \ref{th1} under a significantly weaker assumption, specifically,
 by allowing one of the \ei multiplicities of $\phi(s)$ to be arbitrary (excluding the trivial case of $s$ being central in $G$).
A step in this direction is made in our paper \cite[Theorem 2]{TZ21}, where we prove the following: 

\begin{theo}\label{c99} Let $G$ be a simply connected simple  algebraic group defined over $F$ and
let $s\in G$ be a non-regular non-central semisimple element.
Let $V$ be a non-trivial irreducible $FG$-module with associated representation $\phi$.
Suppose that at most one eigenvalue of $\phi(s)$ is of  multiplicity greater than $1$. Then
one of the \f holds:
\begin{enumerate}[]
\item{\rm{(1)}} $G$ is of Lie type $A_n$, $B_n$ ($p\ne 2$), $C_n$ or $D_n$ and $\dim V=n+1,2n+1,2n,2n$, respectively;
  \item{\rm{(2)}} $G$ is of Lie type $B_n$, $n\geq 3$, $p=2$ and $\dim V = 2n$;

\item{\rm{(3)}} $G$ is of type $A_3$ and $\dim V=6;$

\item{\rm{(4)}} $G$ is of type $C_2$, $p\neq 2$, $\dim V=5$.
 \end{enumerate}\end{theo}

In view of Theorem \ref{c99}, in our consideration of semisimple elements having at most one eigenvalue mutiplicity greater than $1$ on some
 irreducible representation,  we may now restrict our attention to regular elements.
We have the following result.

\begin{theo}\label{re1}
Let $G$ be a  simply connected simple algebraic group defined over $F$, $s\in G$ a regular semisimple element,
and let $\phi $ be a non-trivial \irr with $p$-restricted highest weight $\om$. Suppose that at most one \ei \mult of $\phi(s)$ is greater
than $1$. Then  one of the following holds:\begin{enumerate}[]
  \item{\rm{(i)}} $s$ is strongly regular;

\item{\rm {(ii)}} $G$ and $\om$ are as in one of the cases $(2)$ - $(10)$ of the conclusion of Theorem {\rm \ref{th1}};

\item{\rm{(iii)}} $G=A_n$ and $\om\in\{2\om_1,2\om_2\}$, or $(n,p) = (2,3)$ and $\om=\om_1+\om_2$;
 
\item{\rm{(iv)}} $G=B_n$, $n\geq 3$,  $p=2$ and $\om = \om_2$;
 \item{\rm{(v)}} $G=C_n$, $n\geq 3$, $(n,p)\neq (3,3)$ and $\om=\om_2$, or $n=4$, $p\neq 2,3$ and $\om=\om_4;$

\item {\rm{(vi)}} $G=F_4$, $p\neq 3$ and $\om=\om_4$, or $p=2$ and $\om=\om_1$.

\end{enumerate}
 \end{theo}

The following result shows that the exceptions in Theorems \ref{th1} and \ref{re1} are genuine.

\begin{theo}\label{rr4} Let $G$ be as  above and let $\phi$ be a non-trivial \irr of $G$.

  $(1)$ Suppose that  $\phi$ is as in items $(2)$ - $(10)$ of Theorem {\rm \ref{th1}}. Then there exists a regular, not strongly regular,
  semisimple element  $s\in G$ such that all \ei multiplicities of $\phi(s)$ are equal to $1$.

  $(2)$ Suppose that  $\phi$ is as in items ${\rm{(ii)}}$ -  ${\rm{(vi)}}$ of
  Theorem {\rm \ref{re1}}. Then there exists a regular, not strongly regular,  semisimple element  $s\in G$ such that at most one
  \ei multiplicity of $\phi(s)$ is greater than $1$. 
\end{theo}

There are many examples of \ir \reps  $\phi$ and regular elements $s\in T$ such that  $\phi(s)$ has exactly one \ei of \mult greater than 1.
 However, these examples are within a rather narrow class of representations, as the following result shows.

 \begin{theo}\label{ag8} {\rm \cite[Theorem 2]{TZ21}} Let $G$ be a simply connected simple  algebraic group defined over $F$ and $\phi$ a
   non-trivial \irr of $G$. Then
  the following statements are equivalent:\begin{enumerate}[]
   
\item{\rm{(i)}} There exists a non-central semisimple element $s\in G$ such that at most one \ei \mult of $\phi(s)$ is equal to $1$.
\item{\rm {(2)}} All non-zero weights of $\phi$ are of \mult $1$.
  \end{enumerate}
\end{theo}

\noindent The \ir \reps $\phi$ whose non-zero weights  are of \mult $1$ have been determined in \cite{TZ2}
and those with $p$-restricted highest weight are reproduced in Tables~\ref{tab:omega1} and \ref{tab:omega2}; see the end of the manuscript.

One can ask whether the assumption that $\om$  be $p$-restricted can be dropped in Theorems \ref{th1} and \ref{re1}.
 Here, we have the following result.

 \begin{theo}\label{td4}
Let $G$, $s$ be as in Theorem {\rm \ref{th1}}, where ${\rm char}(F)=p> 0$.
Let $\phi$ be an irreducible representation of $G$ with \hw  $\om=\sum p^i\lam_i$, where all $\lam_i$ are $p$-restricted weights and at least two of them are non-zero. Suppose that 
at most one of the \ei multiplicities  of $\phi(s)$ is greater than $1$.
Then either $s$ is strongly regular or for each $i$ with $\lambda_i\ne 0$, we have $(G,\lambda_i) = (G,\omega)$ for $(G,\omega)$ as in items $(2)-(10)$ of
Theorem {\rm \ref{th1}}. Moreover, \begin{enumerate}[]
\item{\rm{(i)}} if  $(G,p) = (C_n,2)$, then for all $i$, $(\lam_i,\lam_{i+1})\neq (\om_n,\om_1)$;
\item{\rm{(ii)}} if $(G,p) = (G_2,2)$, then for all $i$, $(\lam_i,\lam_{i+1})\neq (\om_1,\om_1)$;
  \item{\rm{(iii)}} if $(G,p) = (G_2,3)$, then  for all $i$, $(\lam_i,\lam_{i+1})\neq (\om_2,\om_1) $.\end{enumerate}\end{theo}

\vskip1cm

{\bf Notation} We fix throughout an algebraically closed field $F$ of characteristic $p\geq 0$. By abuse of notation, when we write $p>p_0$ for some prime $p_0$, we allow
the case ${\rm char}(F) = 0$.

 Throughout the paper  $G$ is a  simple simply connected linear algebraic group defined over $F$. (As mentioned before, we will suppress
 the adjective ``linear''.)
 All $G$-modules considered are rational finite-dimensional $FG$-modules.

We fix a maximal torus $T$ in $G$, which in turn defines the roots of $G$ as well as the weights of
$G$-modules and representations. Recall that the $T$-weights of a $G$-module $V$ are the \ir constituents of the restriction of $V$ to $T$.
When $T$ is fixed (which we assume throughout the paper), we omit the reference to $T$ and write ``weights'' in place of ``$T$-weights''. The set of
weights of a $G$-module $V$ is denoted by $\Omega(V)$. Recall that the roots of $G$ are the non-zero weights of the adjoint representation,
equivalently of the $G$-module ${\rm Lie}(G)$.  We denote the set of roots by $\Phi$ or $\Phi(G)$. The $\mathbb Z$-span of $\Phi$ is
denoted by $R$ or $R(G)$. The weights in $R$ are called {\it radical}. We fix a base of the root system and the 
simple roots of $\Phi(G)$ are denoted by $\al_1\ld \al_n $ and ordered as in \cite{Bo}; we denote the associated set of positive roots by
$\Phi^+$ and the $\mathbb Z_{\geq 0}$-span of the simple roots by $R^+$.

The Weyl group of $G$ is denoted by $W$ (or $W(G)$); in fact $W(G)=N_G(T)/T$ so the conjugation
action of $N_G(T)$ on $T$ yields an action of $W(G)$  on $T$ and on the $T$-weights. The $W$-orbit of $\mu\in\Om$ is denoted by $W\mu$.
Minuscule weights are defined and tabulated in \cite[Ch. VIII, \S 7.3]{Bo8}.

The set $\Omega={\rm Hom} (T,F^\times)$ (the rational homomorphisms of $T$ to  $F^\times$)
is called the {\it weight lattice}, which is a free $\ZZ$-module of finite rank
called {\it the rank of} $G$.  For a $G$-module $V$, $t\in T$ and $\mu\in \Om(V)$, $\mu(t)$ is an \ei of $t$ on $V$; then $(-\mu)(t)=t\up$
and $(\mu+\nu)(t)=\mu(t)\nu(t)$ for $\nu\in\Om(V)$.
As every semisimple element of $G$ is conjugate to an element of $T$, in most  results which follow, the assumption that $s$
is semisimple   is  replaced by the assumption that $s\in T$.

We let $\{\om_1,\dots,\om_n\}$ be the set of fundamental dominant weights with respect to the choice of simple roots. 
These form a $\ZZ$-basis of $\Omega$, so every $\nu\in \Omega$ can be expressed in the form $\nu = \sum a_i\om_i$ ($a_i\in \ZZ$); the set of  $\nu$ with $a_i\geq 0$ is denoted by
$\Omega^+$. If $p>0$, we set $\Omega^+_p$ to be the set of weights $\nu = \sum a_i\om_i$ with $0\leq a_i<p$. If ${\rm char}(F)=0$ we let $\Om_p^+ = \Om^+$. 
The weights of $\Omega^+$, respectively $\Omega^+_p$ are called {\it dominant}, respectively $p$-{\it restricted} weights (so that if ${\rm char}(F) = 0$, all
dominant weights are by definition $p$-restricted weights).
 An \ir $G$-module is called $p$-{\it restricted} if its highest weight is $p$-restricted.

There is a standard partial ordering of elements  of $\Omega$; specifically, for $\mu,\mu'\in\Omega$ we write
$\mu\prec\mu'$ and $\mu'\succ \mu$ \ii $\mu'-\mu\in R^+$ and $\mu\ne \mu'$. We write $\mu\preceq\mu'$ or $\mu'\succeq\mu$  if and only
if $\mu'-\mu\in R^+$.   For $\mu,\nu\in\Omega^+$, with
$\mu\preceq \nu$, we say $\mu$ \emph{is subdominant to} $\nu$.  Every \ir $G$-module $V$ has a weight $\om$ such that $\mu\prec\om$ for every $\mu\in\Omega(V)$ with $\mu\neq \om$.
This is called the {\it highest weight of} $V$. There is a bijection between $\Omega^+$ and the set of isomorphism classes of irreducible $G$-modules,
so for $\om\in\Omega^+$ we denote by $V_\om$ the \ir $G$-module with highest weight $\om$.
The zero weight space of $V$ is sometimes denoted by $V^T$. In general, for $S\subseteq G$ we set $V^S=\{v\in V: sv=v$ for all $s\in S\}$.
The maximal root of $\Phi(G)$ is denoted by $\om_a$; this is the highest weight
of the adjoint module ${\rm Lie}(G)$ and of a composition factor $V_a$ of ${\rm Lie}(G)$.

Because of natural isomorphisms between certain indecomposable root systems, we will assume throughout that the rank of $G$ is as
in Table~\ref{tab:rank} below.
\begin{table}[h]
$$\begin{array}{|l|c|}\hline
G& {\rm rank}(G) \\
\hline
A_n&n\geq 1\\
\hline
B_n&n\geq 3 \\
\hline
C_n&n\geq 2\\
\hline 
D_n&n\geq 4\\
 \hline

\end{array}$$
\caption{{\rm rank}(G)}\label{tab:rank}
\end{table}
For brevity we write $G=A_n$ to say that $G$ is a simple simply connected algebraic group of type $A_n$, and similarly
for the other types. For classical groups $G$ the module with highest weight $\om_1$ is called {\it the natural module}, with one exception,
namely the natural module for $G=B_n$ when $p=2$ is a reducible $(2n+1)$-dimensional module with two composition factors $V_{\om_1}$ and $V_0$.

In our proofs, we regularly use so-called ``Bourbaki weights'', which are elements of a $\ZZ$-lattice containing $\Omega$ with basis   $\ep_1,\ep_2,...$;
the explicit expressions of the fundamental weights and the simple roots of $G$ in terms of the $\ep_i$ are given in \cite[Planches 1 -- XIII]{Bo}.

 If $h: G\ra G$ is a surjective algebraic group homomorphism  and $\phi$ is a representation of $G$ then the $h$-twist $\phi^h$ of $\phi$ is defined as the mapping $g\ra \phi(h(g))$ for $g\in G$. Of fundamental importance
 is the Frobenius mapping $Fr: G\ra G$ arising from the mapping $x\ra x^p$ $(x\in F)$ when $p>0$. If $V$
 is a  $G$-module then, for integers $k\geq 0$, the modules $V^{Fr^k}$ are called {\it Frobenius twists of} $V$; if $V$ is \ir with highest weight $\om$ then the
 highest weight of
 $V^{Fr^k}$ is $p^k\om$. Every $\om=\sum a_i\om_i\in\Omega^+$ has a unique $p$-adic expansion
 $\om=\lam_0+p\lam_1+\cdots +p^k\lam _k$ for some $k$, where $\lam_0\ld \lam_k\in\Omega_p^+$.
 The Steinberg tensor product theorem then implies that  $V_\om\cong V_{\lam_0}\otimes V_{\lam_1}^{Fr}\otimes\cdots \otimes V_{\lam_k}^{Fr^k}$.

Finally, we recall that  a matrix $M\in M_{n\times n}(F)$ is called \emph{cyclic} if it is diagonalisable and all \ei multiplicities equal $1$; we will say $M$ is \emph{almost cyclic}, if it is diagonalizable and 
at most one \ei \mult is greater than $1$. This leads us to introduce the following terminology to be used throughout.
\begin{defi}\label{def:ac} Let $s\in T$ and let $V$ be a $G$-module with associated representation $\phi:G\to {\rm GL}(V)$. We say that $s$ is \emph{cylic on $V$}, respectively
  \emph{almost cyclic on $V$} if $\phi(s)$ is cyclic, respectively almost cyclic. 

  \end{defi} 



\section{Preliminaries}

We start by recalling a characterization of regular semisimple elements, as mentioned in the Introduction. We will use this
frequently without direct reference. 

\begin{lemma}\label{re3} {\rm \cite[Ch. III, \S 1, Corollary 1.7]{Spr}} Let $G$ be a semisimple algebraic group, $T\leq G$ a maximal torus  and $s\in T$.
  Then the \f conditions are equivalent:\begin{enumerate}[]

\item{\rm{(1)}} $s$ is regular;
\item{\rm{(2)}} if $gs=sg$ for $g\in G$, then g is semisimple;
\item{\rm{(3)}} for all roots $\al\in\Phi(G)$, $\al(s)\ne 1$.
  \end{enumerate}\el

The following lemma concerns assertion (*) in the Introduction, rephrased in terms of Definition~\ref{def:ac}.

\begin{lemma}\label{ss2} Let
 $V,M$ be non-trivial G-modules and  $s\in T$.

 $(1)$ If $s$ is  cyclic on $V$, then $s$ is regular.

 $(2)$ If $s$ is  almost cyclic on $V\otimes M$, then $s$ is regular.\el

 \begin{proof} (1) Let $\rho$ be the \rep afforded by $V$. If $s$ is  cyclic on $V$ then
   $C_{\GL(V)}(\rho(s))$ consists of semisimple elements. So  the claim follows from Lemma \ref{re3}(2).

   (2) By \cite[Lemma 1]{TZ21}, $s$ is cyclic on $V$, so the result follows from (1).\end{proof}

\begin{defi}\label{se1} Let $V$ be a $G$-module, $\mu,\nu\in\Omega(V)$, and $s\in T$. We say that $s$
  \emph{separates the weights} $\mu,\nu$ of $V$ if $\mu(s)\neq\nu(s)$. If this holds for every pair of weights
  $\mu,\nu\in\Om(V)$, we say that $s$ \emph{separates the weights of} $V$. \end{defi}

If $s$  separates the weights of $V$ then the \ei multiplicities of $s$ acting on $V$ are simply the weight
multiplicities of $V$.

 \begin{lemma}\label{td2} 
 Let $V,V_1, V_2$ be non-trivial G-modules.
Let $s\in T\setminus Z(G)$ and assume that $s$ is almost cyclic on $V$.\begin{enumerate}[]

\item{\rm{(1)}} \St $V=V_1\otimes V_2$. Then $s$  is cyclic on $V_1$ and on $V_2$, all weights of $V_1$ and $V_2$ are of multiplicity $1$, and $s$ is regular.

\item{\rm{(2)}}Suppose that $\Om(V_1)+\Om(V_2)=\Om (V)$. Then $s$ separates the weights of $V_i$, for $i=1,2$.\end{enumerate}

\el

\begin{proof}  This follows from \cite[Lemma 1]{TZ21} and \cite[Lemma 3]{TZ21}.\enp

The set of weights of a $G$-module $V$ determines the action of a given semisimple element on the
  module $V$. We will rely on the following important result \cite[Theorem 1]{Pr} describing precisely, in many cases, the set $\Omega(V)$. We require an additional notation, defined in {\it{loc.cit.}}; we write $e(G)$ for the maximum of the squares of the ratios of the lengths of the roots in $\Phi(G)$. So $e(G)=1$ for $G$ of
  types $A_n,D_n,E_n$, $e(G)=2$ for $G$ of type $B_n,C_n,F_4$ and $e(G)=3$ for $G$
of type $G_2$.

\begin{theo}{\cite[Theorem 1]{Pr}}\label{premet}  Assume that $p=0$ or $p>e(G)$. Let $\lambda$ be a $p$-restricted weight. Then
$\Omega(V_\lambda)=\{w(\mu)\ |\ \mu\in \Omega^+, \mu\preceq\lambda,w\in W\}.$\end{theo}

As an immediate Corollary of Theorem~\ref{premet}, we have:

\begin{lemma}\label{wtlattice}
  Assume that $p=0$ or $p>e(G)$. Let $\lambda,\mu\in \Omega^+$  with $\lambda$ $p$-restricted, and $V_\lambda$,
  respectively, $V_\mu$ the associated irreducible $G$-modules. Then the following hold.\begin{enumerate}[]
\item{\rm{(1)}} If $\mu\prec\lambda$ then    $\Omega(V_\mu)\subseteq \Omega(V_\lambda)$.
\item{\rm{(2)}} If $\lambda+\mu$ is $p$-restricted then   $\Omega(V_{\lambda+\mu})=\Omega(V_\lambda\otimes V_\mu)=\Omega(V_\lambda)+\Om( V_\mu)$.
\item{\rm{(3)}} If $\lam$ is radical then some root is a weight of $V_\lam$; otherwise $\Omega(V_\lam)$ contains some minuscule weight.\end{enumerate}\end{lemma}

We begin our considerations of strongly regular elements and their action in certain representations by considering the action on $V_a$, the
irreducible $G$-module with highest weight the highest root.

  \bl{sa1}  Assume that $p=2$ if $G=A_1$, $p\ne 3$ if $G=A_2$ or $G_2$, and $p\ne 2$ if $G=B_n$,  $C_n$ or $F_4$. Let $s\in T\setminus Z(G)$. Then $s$ is strongly
    regular \ii  $s$ is almost cyclic on $V_a$.\el

\bp   If $G=A_1$ and $p=2$, then all non-identity semisimple elements of $G$ act cycically on $V_a$ (which is just a twist of the natural representation of $G$) and all
non-identity semisimple elements are strongly regular. So here the statement is clear.

Now turn to the other cases and suppose that  $s$ is strongly regular, so that $\alpha(s)\ne \beta(s)$ for roots $\alpha,\beta\in\Phi(G)$, with $\alpha\ne \beta$.
Then, with one exception, the hypothesis of (1), together with Theorem~\ref{premet}, implies that $\Omega(V_a) = \Phi(G)\cup\{0\}$ and so $s$ is almost cyclic on $V_a$. The exception occurs when $G=G_2$ and $p=2$. But here as well, we have $\Om(V_a) = \Phi(G)\cup\{0\}$; see for example \cite{Lu1}. 

Now suppose that $s$ is almost cyclic on $V_a$ (and ${\rm rank}(G)>1$).  Then, in all cases considered here, the zero weight has multiplicity at least 2 on $V_a$.
Thus $\alpha(s)\ne \beta(s)$ for all roots $\alpha,\beta\in\Phi(G)$, with $\alpha\ne \beta$. That is, $s$ is strongly regular.\end{proof}

\begin{rem}\label{rr3} {\rm The exceptions in
    Lemma ~\ref{sa1} are genuine.  Indeed, if $p\ne 2$ and $G={\rm SL}_2(F)$, then elements of $T$ which are not strongly regular are of the form $s={\rm diag}(c,c^{-1})$ with $c^4=1$,
since $\alpha(s) = c^2$. We choose $c$ such that $c^2=- 1$ and then $s$ is non-central, regular,  and almost cyclic on $V_a$.
If $G=C_n, n\geq 2$, with $p=2$,
then $\om_a=2\om_1$. Then $V_a$ is the Frobenius twist of $V_{\om_1}$, so $s$ is almost cyclic  on $V_a$  \ii $s$ is almost cyclic
$V_{\om_1}$.
However, it is easy to construct non-regular elements $s\in T$  that are almost cyclic  on $V_{\om_1}$.  For $G=G_2$, $p=3$, see Proposition~\ref{2g3}(4), and for
$G=F_4$, $p=2$, see Remark \ref{nr2}(1). For $G=A_2$ with $p=3$ take $s=\diag(a,-a,-a^{-2})$ for $a\in F^\times$, so $\al_1(s)=(-\al_1)(s)$. Then the Jordan normal
form of $s$ on $V_a$ is $\diag(1,-1,-1,a^3,-a^3,a^{-3},-a^{-3}).$ This matrix is almost cyclic if $a^{4}\neq 1$. Finally, for $G=B_n,n\geq3,$ and $p=2$, see Lemma~\ref{sr3} and Remark \ref{b23} }.\end{rem}

Lemma~\ref{ss4} and Remark~\ref{r22}  will not be required in what follows, but give some additional information about the relationship between
a semisimple element being regular and its action on $V_a$.

\begin{lemma}\label{ss4} Let $s\in T$. Let $\om$ be the highest short root and $V = V_\om$.\begin{enumerate}[]
  \item{\rm{(1)}} If $s$ is regular, then $V^s = V^T$ and 
 $V_a^s=V_a^T$.
\item{\rm{(2)}} Suppose that $V_a^s = V_a^T$. If $p>e(G)$ or if $G=G_2$ and $p=2$, then $s$ is regular.
  \end{enumerate}\el

 \begin{proof} By Lemma \ref{re3}, $s$ is regular \ii $\al(s)\neq 1$ for all $\alpha\in\Phi(G)$.
  For (1), we use the fact that all non-zero weights of $V_{a}$ and $V$ are roots.  
   (2)   
   Using Theorem~\ref{premet}, one observes that for $p>e(G)$,  $\Omega(V_a) = \Phi(G)\cup\{0\}$.
   As mentioned in the proof of Lemma~\ref{sa1}, this also is true for
  $G=G_2$, when $p=2$.
  So $V_a^T= V_{a}^s$ implies $\al(s)\neq1$ for  $\al\in\Phi(G)$, and hence $s$ is regular.\end{proof}

\begin{remk}\label{r22}{\rm If $p=e(G)$, there exists a non-regular element $s\in T$, such that $V_a^s = V_a^T$.  
  See Proposition~\ref{2g3}(5), for $G=G_2$ with $p=3$ and Remark~\ref{nr2}(2) for $G=F_4$ with $p=2$.  Now let $G=C_n$ with $p=2$. Then
  $V_a=V_{2\om_1}$ and hence 
  $\Om(V_a) = \{\pm 2\ep_1,\ldots ,\pm 2\ep_n\}$. So $V_a^T=\{0\}$, and $ V_a^s=V_a^T$ \ii $ V_a^s=\{0\}$, equivalently, \ii 1 is not an \ei of $s$ on $V_{\om_1}$.
  As
  $n\geq 2$, there are non-regular elements $s\in T$ such that $V_{\om_1}^s = \{0\}$. 
 For $G=B_n$, when $p=2$, the non-zero weights of $V_a$ are $\{\pm\ep_i\pm \ep_j\ |\ 1\leq i<j\leq n\}$. Choose $s$ such that $\ep_1(s)=1$, so $s$ is not regular, but $(\ep_i\pm\ep_j)(s)\ne 1$ for all $i\ne j$. }\end{remk}


\section{Classical groups}

In this section we prove Theorems~\ref{th1}, \ref{re1} and \ref{rr4}  for classical groups.
In view of Theorem \ref{ag8}, we must examine the modules $V_\om$, for $\om\in\Om_p^+$,  all of whose
non-zero weights are of \mult 1. These are listed in Tables~\ref{tab:omega1} and \ref{tab:omega2}.
Recall that we have fixed a maximal torus $T$ of $G$, and we assume that our semisimple element belongs to $T$.

\subsection{Groups of type $A_n$}\label{sec:An}

Throughout this subsection, $G$ is of type $A_n$.

\bl{ww2} Let $\om\in\Om^+$ and $n\geq 2$.\begin{enumerate}[]

\item{\rm{(1)}} If $\om\succ 2\om_1$ then $\om\succeq \om_1+\om_2+\om_n$.
\item{\rm{(2)}} If $\om\succ \om_1+\om_n$, then $\om\succeq \om_2+\om_{n-1}$ if $n\geq 3$ and $\om\succeq \nu$, for some
$\nu\in\{ 3\om_1, 3\om_2\}$ if $n=2$.\end{enumerate}\el

\bp Let $\om\succ 2\om_1$. By \cite[Lemma 3.2]{Z18}, we can assume that $\om=2\om_1-\om_{j-1}+\om_j+\om_k-\om_{k+1}$ for
some $j, k\in\{1\ld n\}$, $j\leq k$.  (We set $\om_0=0=\om_{n+1}$.) \itf $k=n$ and $j\leq2$, and in both cases (1) easily follows.

Now, let $\om\succ\omega_1+\omega_n$. As above,  we can assume that $\om=\om_1+\om_n-\om_{j-1}+\om_j+\om_k-\om_{k+1}$.
Then $j\in\{1,2\}$ and $k\in\{n-1,n\}$. If $n\geq 3$, then $\om\in\{2\om_1+2\om_n,\om_2+\om_{n-1 },\om_2+2\om_n, 2\om_1+\om_{n-1} \}$. As
$2\om_1+2\om_n\succ \om_2+2\om_n\succ \om_2+\om_{n-1}$ and $2\om_1+\om_{n-1}\succ \om_2+\om_{n-1}$, the result  follows. If $n=2$ then $\om\in\{2\om_1+2\om_2,
3\om_2, 3\om_1 \}$ and $2\om_1+2\om_2\succ 3\om_i$, $i=1,2$, and the result of (2) follows.\end{proof}

\begin{propo}\label{ac9} Let $G=A_n$, with $n\geq 1$, $s\in T_{reg}$ and let 
$\om\in\Om^+_p$.\begin{enumerate}[]
  \item{\rm{(1)}} Assume that $\om\notin\{ 0,\om_1\ld \om_n, 2\om_1,2\om_n\}$ and that $s$ is almost cyclic on $V_\om$. 
Then $s$ is strongly regular, 
 unless   $(n,p)=(2,3)$ and $\om=\om_1+\om_2$.

\item{\rm{(2)}} If $\om\in\{ 2\om_1,2\om_n\}$ and $s$ is cyclic on $V_\om$, then $s$ is strongly regular.
\item{\rm{(3)}} If $(n,p) = (2,3)$ and $\om=\om_1+\om_2$, with $s$ cyclic on $V_\om$, then $s$ is strongly regular.\end{enumerate}\end{propo}

\bp  Set $V=V_\om$. Recall that the $T$-roots of $G$ are $\ep_i-\ep_j$ for $i,j\in\{1\ld n+1\}$, $i\ne j$. Then $s$ regular implies that
$\alpha(s)\ne 1$ for all $\alpha\in\Phi(G)$. 

For (1), suppose the contrary, that is, $s$ is almost cyclic on $V_\om$ and $s$ is not strongly regular, so that  $\al(s)= \be(s)$ for some roots $\al\ne\be$.
As $W$ acts transitively on the roots, we can assume that $\al=\al_1=\ep_1-\ep_2$.

If $n=1$ then  $\be=-\al_1$ and $\al(s)= \be(s)$ implies $(2\al_1)(s)=1$, whence $\al_1(s)=-1$. Let $\om=a\om_1$, where by hypothesis
$3\leq a\leq p-1$. Then the weights of $V$ are $\om-i\al_1$ for $0\leq i\leq a $, and $(\om-i\al_1)(s)=(-1)^i\om(s)$, contradicting that $s$ is
almost cyclic on $V$.

Now let $n\geq 2$. Recall that $W$ acts on $\ep_1\ld \ep_{n+1}$ as the symmetric group $S_{n+1}$ does. Since $\be=\ep_i-\ep_j$ for some $i\neq j$, we can assume
that
$\be\in\{-(\ep_1-\ep_2),\pm(\ep_1-\ep_3),\pm(\ep_2-\ep_3),\ep_3-\ep_4\}$
 (where $\ep_3-\ep_4$ occurs
 only if $n\geq 3$). As $s$ is regular, $\al-\be$ is not a root, so $\be\notin\{\ep_1-\ep_3, -\ep_2+\ep_3\}$,  and hence
 $\be\in\{-(\ep_1-\ep_2),-(\ep_1-\ep_3),\ep_2-\ep_3,\ep_3-\ep_4\} = \{-\alpha_1,-\alpha_1-\alpha_2,\alpha_2,\alpha_3\}$. In addition,
  as $s$ is regular,  if $\be=-\al$ then $p\neq 2$ and $\al_1(s)=-1$.

  By   \cite[Lemma 11]{TZ21} (see Definition 3, {\it{loc.cit.}}) and Theorem \ref{premet}, there exists a weight
  $\mu\in \Omega(V)\cap\{\om_1+\om_i,\om_i+\om_n: i=1\ld n\}$ 
  and all weights subdominant to $\mu$ also lie in $\Omega(V)$.
  We consider in turn each possibility for  the weight $\mu$.

Suppose first that   $\mu:=\om_1+\om_i$ for $1< i<n$.
 Then   $\mu-\alpha_1-\alpha_2-\cdots-\alpha_i = \om_{i+1}\in\Om(V)$ and $\om_{i+1}-\al$, $\om_{i+1}-\be\in W\mu$. In addition,
 $\om_{i+1}-\al_{i+1}-\alpha$ and $\om_{i+1}-\alpha_{i+1}-\beta\in\Om(V)$, as long as $\beta\in\{-\alpha_1,\alpha_2\}$ when $i=2$.  Since $s$ is almost cyclic on $V$,
 we deduce that
 either $i=2$ and $\beta\in\{-\alpha_1-\alpha_2,\alpha_3\}$ or $\om_{i+1}-\alpha = \om_{i+1}-\alpha-\alpha_{i+1}$, so $\alpha_{i+1}(s)=1$, contradicting the fact that $s$ is regular.  So it remains to consider the
 case where
 $i=2$, $\mu=\om_1+\om_2$, $n\geq 3$, and $\beta\in\{-\alpha_1-\alpha_2,\alpha_3\}$. We give the details for the case $\beta=-\alpha_1-\alpha_2$. Here
 $\eta = \mu-\alpha_1-2\alpha_2-\alpha_3\in\Om(V)$ and
 $\eta-\alpha,\eta-\beta\in\Om(V)$. Hence we deduce as above that $\om_{i+1} = \eta$ and so $(\alpha_2+\alpha_3)(s) = 1$, contradicting that $s$ is regular. The
 case $i=2$ and $\beta=\alpha_3$ is entirely similar.

 Suppose now that $\mu=\omega_1+\omega_n$. Then all roots are weights of $V_\om$, and since $\al(s)=\be(s)$, we have  $(-\al)(s)=(-\be)(s)$,
 whence $ \al(s)=\be(s)=-1$ and
 $(\al+\be)(s)=1$.
Then $\al+\be$ is not a root as $s\in T_{reg}$. Therefore,
$\be\in\{-\alpha_1,\alpha_3\}$.  If $\be=\al_3$, so $n\geq 3$, then $\al_1+\al_2+\al_3$ is a root and $(\al_1+\al_2+\al_3)(s)=\al_2(s)$. As
before, we deduce that $\al_2(s)=\al_1(s)=-1$, and then $(\al_1+\al_2)(s)=1$, whence  a contradiction.
So we are left with the case where $\beta=-\alpha$ and $\al(s)=(-\al)(s)=-1$.  By Lemma \ref{sa1} (and the hypothesis of (1)), we may assume that $\om\neq \om_1+\om_n$.
 By Lemma \ref{ww2}, $\om\succeq \nu$, where
    $\nu\in\{ 3\om_1, 3\om_2\}$ if $n=2$, and   $\nu= \om_2+\om_{n-1}$ if $n\geq 3$.
    As usual, Theorem~\ref{premet} implies that all weights subdominant to $\nu$ lie in $\Omega(V)$.

    Now if $n\geq 3$, so $\nu = \omega_2+\omega_{n-1}$, then set $\eta = \nu-\alpha_1-2\alpha_2-\alpha_3-\cdots-\alpha_{n-1}$. Note that $\eta\pm\alpha_1$ and
    $\eta-\alpha_n\pm\alpha_1$ lie in $\Omega(V)$. But then we deduce that $\alpha_n(s) = 1$, giving the usual contradiction. If $n=2$ and $\nu = 3\omega_1$, we
    argue similarly, setting $\eta = \nu-\alpha_1$, and deduce that $\alpha_2(s) = 1$. The case $\nu=3\omega_2$ is analogous.

    Note that the cases considered above cover also the weights $\mu = \omega_i+\omega_n$, for $2\leq i\leq n$, by considering the dual $V^*$; so (again invoking
    duality), it remains to consider the case $\mu = 2\omega_1$. Then Lemma~\ref{ww2} shows that $\omega\succeq \omega_1+\omega_2+\omega_n$, which then implies
    that  $\omega\succ \omega_3+\omega_n$, if $n\geq 3$, and $\omega\succeq \omega_1+2\omega_2$, if $n=2$. In the former case the result follows from previously
    considered cases if $n\geq 4$.  For the case $n=3$ and $\omega\succeq \omega_1+\omega_2+\omega_3$, we consider the different possibilities for $\beta$,
    as before; recall that $\alpha = \alpha_1$ and $\beta\in\{-\alpha_1,-\alpha_1-\alpha_2,\alpha_2,\alpha_3\}$. Setting $\nu = \omega_1+\omega_2+\omega_3$, and
    $\eta = \nu-\alpha_1-\alpha_2-\alpha_3$ and $\eta' = \nu-\alpha_1-2\alpha_2-\alpha_3$, we have that
    $\eta-\alpha, \eta-\beta,\eta'-\alpha,\eta'-\beta\in\Om(V)$ and deduce that $\eta(s) = \eta'(s)$ and so $\alpha_2(s) = 1$ giving the usual contradiction. Finally, let $n=2$, and set $\nu = \omega_1+2\omega_2$ and assume that $\omega\succeq\nu$. We argue as in the previous cases. As before, since $s$ is regular and
    $\alpha(s) = \beta(s)$, $\alpha-\beta$ is not a root. So taking $\alpha = \alpha_1$, we have that $\beta\in\{-\alpha_1,\alpha_2,-\alpha_1-\alpha_2\}$.
    Set $\eta = \nu-\alpha_1-\alpha_2$ and $\eta' = \nu-2\alpha_1-2\alpha_2$. Then one checks that $\eta-\alpha, \eta-\beta,\eta'-\alpha,\eta'-\beta\in\Omega(V)$.
    So we deduce  that $\eta(s) = \eta'(s)$ and so $(\alpha_1+\alpha_2)(s) = 1$, a contradiction.

    Turning to the proof of (2), we assume now that $s$ is regular and cyclic on $V = V_\omega$, with $\omega\in\{2\omega_1,2\omega_n\}$. By
    duality, we may assume that $\omega=2\omega_1$; note that $p\ne 2$.  Suppose the contrary, that is, $s$ is not strongly regular, and let $\al,\be$ be as in
    the proof of (1). Here, we only need to find distinct weights $\eta,\eta'\in\Om(V)$ such that $\eta(s) = \eta'(s)$. It is
    straightforward to find such a pair for each choice of $\beta$ (recall that $\alpha=\alpha_1$). For example, if $\beta = -\alpha_1$, then $2\alpha_1(s) = 0$
    and so we may take $\eta = \omega$ and $\eta' = \omega-2\alpha_1$. Or if $\beta = \alpha_2$ (so $n\geq 2$), we take $\eta = \omega-2\alpha_1$ and
    $\eta' = \omega-\alpha_1-\alpha_2$. The remaining cases are left to the reader.

    To conclude, we turn to (3). Here, as $\Om(V) = \{0\}\cup\Phi(G)$, and $s$ is cyclic on $V$, we see that $\alpha(s)\ne \beta(s)$ for all
    $\alpha,\beta\in\Phi(G)$, $\alpha\ne \beta$; thus $s$ is strongly regular as claimed.\end{proof}

  \begin{corol}\label{cor:an} Theorem~$\ref{th1}$ and Theorem~$\ref{re1}$ hold for $G=A_n$.

  \end{corol}

  \begin{proof} For Theorem~\ref{th1}, we assume that $s\in T$ is cyclic on $V_{\omega}$, for some $\om\in\Om_p^+$,
    $\om\ne 0$.  In addition, we assume that either $p=2$ when $n=1$ or  $\om\not\in\{\om_1,\dots,\om_n\}$. For $n=1$, we have $\om=\om_1$ and $s\in T$ is cyclic on $V_\om$ if and
    only if $s\ne 1$, if and only if $s$ is strongly regular, as claimed. So we now assume $n\geq 2$;  Theorem~\ref{c99} implies that $s$ is
    regular and then we use Proposition~\ref{ac9} to see that $s$ is strongly regular.

    For Theorem~\ref{re1}, we assume that $s\in T_{reg}$ is almost cyclic on $V_\om$, for some $\om\in\Om_p^+$, $\om\ne 0$. We suppose further that either $p=2$ if $n=1$ or 
    $\om\not\in\{2\om_1,2\om_n,\om_1,\dots,\om_n\}$. In addition, we assume that if $n=2$ and $\om=\om_1+\om_2$, then $p\ne 3$.
    Now if $n=1$, with $p=2$ and so $\om=\om_1$, then $s$ is almost cyclic on $V_{\om}$ if and only if $s\ne 1$, if and only if $s$ is cyclic on $V_{\om}$, if and only if $s$ is strongly regular. For $n\geq 2$, Proposition~\ref{ac9} gives the result.\end{proof}

We now investigate the weights in Theorem~\ref{th1}(2) and Theorem~\ref{re1}(iii).

\bl{ce1} Let $G=A_n$ and $\omega\in\Omega_p^+$, $\om\ne 0$.
\begin{enumerate}[]

\item{\rm{(1)}} Assume that $p\ne 2$ if $n=1$. If
$\om\in\{\om_1\ld \om_n\}$, then there exists $s\in T$ with $s$ not  strongly regular and $s$ cyclic on $V_\omega$.

\item{\rm{(2)}} If $\om\in\{ 2\om_1,2\om_n\}$,
then there exists  $s\in T_{reg}$, with $s$ not strongly regular and $s$ almost cyclic on $V_\omega$.\end{enumerate}\el

\bp (2) Here we have $p\ne 2$ and the result follows from Remark~\ref{rr3} when $n=1$. So we now take $n\geq 2$. By duality, it suffices to consider $\om=2\om_1$.
Choose $T\leq \SL_{n+1}(F)$ to be the group of diagonal matrices $t=\diag(t_1\ld t_{n+1})$ of determinant $1$, and set
$\ep_i(t)=t_i$,
for $i=1\ld n+1$.

Suppose first that
$n\geq 3$. Let $s=\diag(a^3,a^2,a, b_1\ld b_{n-2})\in \SL_{n+1}(F)$ for some $a,b_1\ld b_{n-2}\in F^\times$.  Then $\al_1(s)=\al_2(s)=a$ so $s$ is not strongly regular.
The weights of $V_\om$ are $2\ep_i$ and $\ep_i+\ep_j$ for $i,j\in\{1\ld n\}$, $i\neq j$. So the \eis of $s$ on $V$ are $a^6,a^4,a^2,b_i^2$
for $i=1\ld n-2$ and $a^5,a^4,a^3,a^kb_i, b_ib_j $ for $1\leq i<j\leq n-2$, $k=1,2,3$. It is straightforward to see that
there exist $a,b_1\ld b_{n-2}\in F^\times$ such that $a^4$ is  the only \ei of $s$ on $V$ of \mult greater than 1, whence the result in this case.

If $n=2$, then let $s=\diag(a,-a ,-a^{-2})$. Then $\al_1(s)=(-\al_1)(s)=-1$ and the \eis of $s$ on $V_{2\om_1}$
are $a^2,a^2,a^{-4}, -a^2,-a^{-1},a^{-1}$, so $s$ is not strongly regular and almost cyclic  on $V_{\om}$ for a suitable choice of $a\in F^\times$. 

(1) A similar construction can be used when $\om\in\{\om_1\ld \om_n\}$. The assumption on the characteristic in case $n=1$ is
necessary as all non-identity semisimple elements in ${\rm SL}_2(F)$ are strongly regular when ${\rm char}(F) = 2$. 
 \enp

 We note that Remark~\ref{rr3} and Lemma~\ref{ce1} establish the statements of Theorem~\ref{rr4} for the groups of type $A_n$.

\subsection{Groups of type $B_n$, $C_n$ and $D_n$}\label{82}

In this section, we prove the main results for classical groups  not of type $A_n$. Recall that we assume that $n\geq 3$ when $G=B_n$,
$n\geq 2$ when $G = C_n$, and $n\geq 4$ when $G = D_n$. 

The weights of $V_{\om_1}$ are $\pm \ep_1\ld \pm \ep_n$ and, additionally, 0 when $G$
is of type $B_n$ and $p\ne 2$. So every  element $s\in T$ can be written as $s=\diag(a_1\ld a_n,1,a_n\up\ld a_1\up)$, with respect to
some basis of $V_{\om_1}$,  where 1 must be dropped if $G= B_n$ with $p=2$.

\bl{dn33}  Let $G=B_n$, $n\geq 3$,  or $D_4$, $ n\geq 4$, or $G=C_n$, $n\geq 3$ and $p=2$. Let $\om\in\Om_p^+$, $\om\ne 0$,  and  $s\in T_{reg}$. Suppose that $s$ is almost cyclic on $V_\om$.
Then one of the following holds:
\begin{enumerate}[]
\item{\rm{(i)}} $s$ is strongly regular,
\item{\rm{(ii)}} $G=B_n$ and $\om\in\{\om_1,\om_n\}$,
  \item{\rm{(iii)}} $G=B_n$, $p=2$ and $\om=\om_2$,
  \item{\rm{(iv)}} $G=C_n$ and $\om\in\{\om_1,\om_2,\om_n\}$, or
  \item{\rm{(v)}} $G=D_n$ and  $\om\in\{\om_1,\om_{n-1},\om_n\}$.  \end{enumerate}\el

\begin{proof} By Theorem \ref{ag8}, all non-zero weights of $V_\om$ are of \mult 1. The non-zero $p$-restricted \ir $G$-modules with this property
  are listed in Tables~\ref{tab:omega1} and \ref{tab:omega2}. 
We find that 
$\om\in\{0,\om_{1},2\om_{1},\om_{2},\om_{n}\}$ for $G=B_n$,
$\om\in\{0, \om_{1},2\om_{1},\om_{2},\om_{n-1},\om_{n}\}$ for $G=D_n$, and $\om\in\{0,\om_1,\om_n,\om_2\}$ for $G=C_n$.
Lemma \ref{sa1} then gives the result unless
$G = B_n$ or $D_n$ and $\om=2\om_1$.

So suppose that $p\ne 2$, $\om=2\om_1$ and $s$ is almost cyclic on $V_{\om}$. Note that $\om_2=2\om_1-\al_1$, so by
Lemma \ref{wtlattice}(1), $\Om(V_{\om_{2}})\subset \Om(V_{2\om_{1}})$. Moreover, the $0$ weight of $V_{2\om_1}$ has multiplicity strictly greater than 1.
(See Table~\ref{tab:omega2}.)
Finally, as $V_a = V_{\om_2}$ and $p\ne 2$, we have $\Phi(G)\subset \Om(V_{2\om_1})$. Now suppose that $s$ is not strongly regular; then Lemma~\ref{sa1}
implies that $s$ is not almost cyclic on $V_{\om_2}$, and since here as well the 0 weight occurs with multiplicity at least $2$, we have that
there exist $\alpha,\beta\in\Phi(G)$, $\alpha\ne \beta$ with $\al(s) = \be(s)\ne 1$. But this then implies that $s$ is not almost cyclic on
$V_{2\om_1}$, giving the result.  \end{proof}

\begin{corol}\label{cor:class} Theorem~$\ref{th1}$ and Theorem~$\ref{re1}$ hold for $G=B_n$, $n\geq 3$, $G = D_n$, $n\geq 4$, and $G=C_n$ with $n\geq 3$ and
  $p=2$. 

\end{corol}

\begin{proof} For Theorem~\ref{th1}, we take $s\in T$ and $\om\in\Om_p^+$, $\om\ne 0$, such that $s$ is cyclic on $V_\om$.
  Then all weights of $V_\om$ occur with multiplicity 1 and so $\om$ is as in Table~\ref{tab:omega1}. As all of these
  weights appear in the conclusion of Theorem~\ref{th1}, the result holds.

  For Theorem~\ref{re1}, we take $\omega$ as before and suppose that $s\in T_{reg}$, with $s$ almost cyclic on $V_\om$.
  Then Lemma~\ref{dn33} gives the result.\end{proof}

We now turn to the consideration of the group $G=C_n$, when $p\ne 2$. Our first result gives some information about the
associated weight
lattice $\Om$.

\bl{2om}  Let $\Phi$ be of type $C_n$ and let $\om\in\Om^+$.
  Suppose that $\om\not\in\{0, 2\om_1\}$ and $\om$ is not a fundamental dominant weight. Then one of the following holds:\begin{enumerate}[]
  \item{\rm{(i)}} $\om\succeq \om_1+\om_2$, 
  \item{\rm{(ii)}} $n\geq 3$ and $\om\succeq \om_1+\om_3$, or
    \item{\rm{(iii)}} $n=2$ and $\om\succeq 2\om_2$.\end{enumerate}\el

  \bp   If $\om$ is not radical (so $\om\succ\om_1$ by \cite[Lemma 12]{TZ21}) then the result is available in {\rm \cite[Lemma 2.2]{z20}}. So suppose now that  $\om$
is radical and set $\om=\sum c_i\om_i$, with $\sum c_i\geq 2$. Consider first the case $n=2$. If $c_2\geq2$ then $\om-2\om_2$ is a radical dominant weight, so $\om\succeq 2\om_2$.
If $c_2=1$ then 
$c_1$ is even and $c_1\geq 2$ by hypothesis, so again  $\om\succ\om-\alpha_1 =  (c_1-2)\om_1+2\om_2\succeq 2\om_2$.
So finally, suppose that $c_2=0$, so that $c_1\geq 4$ is even. Then $\om\succeq\om-2\alpha_1= (c_1-4)\om_1+2\om_2\succeq 2\om_2$.

Consider now the case $n\geq 3$.
If $c_1c_3\neq 0$ then $\om-\om_1-\om_3$ is a dominant radical weight, so $\om-\om_1-\om_3\succeq0$,
whence the result. Next suppose that $c_ic_j\neq 0$ for some odd $i<j$, $(i,j)\ne (1,3)$.
As $\om_i\succ\om_{i-2}$ for $i\geq2$, it follows that $\om\succ \om'=\om  +\om_1+\om_3-\om_i-\om_j$, and $\om'$ is a dominant radical weight 
having non-zero coefficient of  $\om_1$ and $\om_3$. By the previous argument, $\om'\succeq \om_1+\om_3$,
whence the result in this case. The same argument works if $c_i\geq 2$ for some odd $i\geq 3$, as $\om_{i-2} = \om_i-\beta$ for some 
$\beta\in\Phi^+$, and so $c_i\om_i = (c_i-1)\om_i + \om_i = (c_i-1)\om_i + \om_{i-2} +\beta$, and we are in the previous case. We are left with the case where $c_i=0$ for all odd $i$ and either $c_i\geq 2$ for some even $i$ or $c_ic_j\ne 0$ for some even $i<j$. In each case, it is a strightforward check to see that
$\om\succeq 2\om_2\succ \om_1+\om_3$ as required.\end{proof}

 \begin{propo}\label{cnn3}
  Let $G=C_n$, $n\geq 3,$ $p\neq 2$, and $s\in T$. Let $\om\in\Om_p^+$, $\om\ne 0$,  and suppose that $s$ is almost cyclic  on $V_{\om}$. Then  $s$ is strongly regular
  unless one of the following holds:\begin{enumerate}[]
  \item{\rm{(1)}} $\om\in\{\om_1,\om_2\}$;
    \item{\rm {(2)}} $n=3,4$ and $\om=\om_n;$ or
\item{\rm {(3)}} $n\geq 4$, $p=3$, and $\om\in\{\om_{n-1},\om_n\}$. \end{enumerate}\end{propo}

\bp We assume that $\om\not\in\{0,\om_1\}$. Then by Theorem \ref{c99}, $s$ is regular, so  $\al(s)\neq 1$ for every root $\al$. By Theorem~\ref{ag8},
$\om$ occurs in Table~\ref{tab:omega1} or \ref{tab:omega2}. The case $\om=2\om_1=\om_a$
is settled in Lemma \ref{sa1}, so we assume as well that $\om\ne 2\om_1$. The other entries of Tables~\ref{tab:omega1} and \ref{tab:omega2} either appear in cases (1) - (3) of the statement
or $p>3$ and  $\om\in\{\om_{n-1}+\frac{p-3}{2}\om_n,~\frac{p-1}{2}\om_n\}$. So consider these remaining configurations. By Theorem~\ref{2om}, either   $\om\succeq\om_1+\om_2$ or
$\om\succeq \om_1+\om_3$. By Lemma~\ref{premet}, we have $\Om(V_{\om_1+\om_2})\subseteq \Om(V)$, respectively $\Om(V_{\om_1+\om_3})\subseteq \Om(V)$

Suppose that $s$ is not strongly regular, so that
we have $\al(s)=\be(s)$ for some $\al,\be\in \Phi(G)$, $\al\neq \be$.
 The roots of $G$ are $\pm 2\ep_i,\pm \ep_j\pm \ep_k$ for $i,j,k\in\{1\ld n\}$. As $s$ is regular,
 we have $\ep_i(s)\neq \pm 1$ and $(\ep_j\pm\ep_k)(s)\ne 1$, for all $i$ and $ j\ne k$. Using the Weyl group, we can assume that $s$ satisfies one of:
 \begin{enumerate}[(i)]
 \item $(2\ep_1)(s)=( 2\ep_2)(s)$,
 \item $(2\ep_1)(s)=( \ep_2-\ep_1)(s)$,
 \item $(2\ep_1)(s)=(\ep_2- \ep_3)(s)$, 
 \item $n\geq 4$ and $(\ep_1- \ep_2)(s)=(\ep_3- \ep_4)(s)$, or
   \item $(2\ep_1)(s) = (-2\ep_1)(s)$.
   \end{enumerate}
We now consider the two cases  $\om\succeq\om_1+\om_2$ or
 $\om\succeq \om_1+\om_3$ separately.

 Suppose first that $\om\succeq\om_1+\om_2\succ \om_1$. Then $V_\om$ has weights $\pm\ep_i$ ($i=1\ld n$) and $\pm2\ep_i\pm\ep_j$, for $1\leq i\ne j\leq n$.
 If (i) holds, $(2\ep_1)(s)=( 2\ep_2)(s)$, so $(2\ep_1\pm\ep_3)(s)=(2\ep_2\pm\ep_3)(s)$.
   Since $s$ is almost cyclic on $V_\om$,  we have $(2\ep_1+\ep_3)(s)=(2\ep_1-\ep_3)(s)$, whence $\ep_3(s)=\pm 1$, contradicting $s$ regular.
   If (ii) holds, $(2\ep_1)(s)=(\ep_2-\ep_1)(s)$ , so that $(2\ep_1-\ep_2)(s) =(-\ep_1)(s)$. As $\pm(2\ep_1-\ep_2)$, $\pm\ep _1$ are weights of $V_\om$,
   it follows that  $(-\ep_1)(s)=-1=\ep_1(s)$,
   again  a contradiction. If (iii) holds, $(2\ep_1)(s)=(\ep_2- \ep_3)(s)$, so that $(2\ep_1-\ep_2)(s) =(-\ep_3)(s)$ and we get a contradiction as in
   (ii).  Note that $\om_1+\om_2\succ \om_3$, so $V_\om$ contains the weights of $V_{\om_3}$. These include the
   weights
   $\pm(\ep_1-\ep_3+\ep_4)$. If (iv) holds, then $(-\ep_1+\ep_3-\ep_4)(s)=(-\ep_2)(s)$, and $(+\ep_1-\ep_3+\ep_4)(s)=\ep_2(s)$; as usual, $s$ almost cyclic on $V_\om$ implies
   that
   $\ep_2(s)=(-\ep_2)(s)$,  again a contradiction. Finally, if (v) holds, so that $(2\ep_1)(s) = (-2\ep_1)(s)$, then $(2\ep_1+\ep_2)(s) = (-2\ep_1+\ep_2)(s)$ and
   $(-2\ep_1+\ep_3)(s) = (2\ep_1+\ep_3)(s)$, which then implies that $\ep_2(s) = \ep_3(s)$ so $(\ep_2-\ep_3)(s)=1$, contradicting that $s$ is regular. 

   Suppose now that $\om\succeq \om_1+\om_3\succ 2\om_1\succ\om_2$. Then $\Omega(V_\lambda)\subseteq \Omega(V_\om)$,  for
   $\lambda\in\{\om_1+\om_3,2\om_1,\om_2\}$;
in
particular,
$\pm2\ep_i\pm\ep_j\pm\ep_k$ are weights of $V_{\om}$ for $i,j,k$ distinct, as are $\pm2\ep_i$ for $1\leq i\leq n$.
If (i) holds then $(2\ep_1)(s)=( 2\ep_2)(s)$ and $(-2\ep_1)(s)=( -2\ep_2)(s)$. As usual, $s$ almost cyclic on $V_\om$ implies that
$(2\ep_1)(s)=(-2\ep_2)(s)$, so that we have $(2\ep_1+\ep_2+\ep_3)(s) =
(\ep_3-\ep_2)(s)$. Similarly, we find that $(-2\ep_1+\ep_2-\ep_3)(s) = (-\ep_2-\ep_3)(s)$ and we deduce that $\ep_3(s) = (-\ep_3)(s)$ contradicting
that $s$ is
regular. 
If (ii) holds, $(2\ep_1)(s)=(\ep_2-\ep_1)(s)$, then $(2\ep_1+\ep_2+\ep_3)(s)=(-\ep_1+2\ep_2+\ep_3)(s)$ and
$(2\ep_1+\ep_2-\ep_3)(s)=(-\ep_1+2\ep_2-\ep_3)(s)$. As $2\ep_1+\ep_2\pm\ep_3$ and $-\ep_1+2\ep_2\pm\ep_3$ are weights of $V_{\om_1+\om_3}$, it
follows that $(2\ep_1+\ep_2+\ep_3)(s)=(2\ep_1+\ep_2-\ep_3)(s)$, whence $(2\ep_3)(s)=1$, a contradiction.
If (iii) holds, $(2\ep_1)(s)=(\ep_2- \ep_3)(s)$, then $(2\ep_1-\ep_2+\ep_3)(s)=1$ and $(2\ep_1-\ep_2-\ep_3)(s)=(-2\ep_3)(s)$. The former equality
means that the \mult of the \ei 1 is greater than 1 and the second equality implies that
$(-2\ep_3)(s)=1$, a contradiction. 
If (iv) holds, $(\ep_1- \ep_2)(s)=(\ep_3- \ep_4)(s)$, $n\geq 4$, then $(2\ep_1-\ep_2-\ep_3)(s)=(\ep_1-\ep_4)(s)$ and
$(-\ep_1+\ep_3-2\ep_4)(s)=
(-\ep_2-\ep_4)(s)$. As  $s$ is almost cyclic on $V_\om$,  we have $(\ep_1-\ep_4)(s)=(-\ep_2-\ep_4)(s)$, whence $\ep_1(s)=(-\ep_2)(s)$, giving the usual
contradiction. Finally, if (v) holds, then $(2\ep_1)(s) = (-2\ep_1)(s)$, both of which occur in $V_{2\om_1}$. We have as well
$(2\ep_1+\ep_2+\ep_3)(s) = (-2\ep_1+\ep_2+\ep_3)(s)$. So we deduce that $(2\ep_1+\ep_2+\ep_3)(s) = (2\ep_1)(s)$ giving $(\ep_2+\ep_3)(s)=1$ and the final contradiction.\end{proof}

\begin{corol}\label{cor:cn1} Theorem~$\ref{th1}$ and Theorem~$\ref{re1}$ hold for $G = C_n$, $n\geq 3$ and $p\ne 2$.\end{corol}

\begin{proof} We first let $s\in T$ and $\om\in\Om_p^+$, $\om\ne 0$ such that $s$ is cyclic on $V_\om$. Then all weights of $V_\om$ occur
  with multiplicity 1 and so $\om$ appears in Table~\ref{tab:omega1}. In addition, Proposition~\ref{cnn3} applies. So either $s$ is strongly regular or one of the following holds:\begin{enumerate}[(i)]
  \item $\om=\om_1$,
  \item $\om  = \om_2$ and $(n,p)=(3,3)$,
  \item $n=3$ and $\om = \om_3$, or
  \item $n\geq 4$, $p=3$ and $\om\in\{\om_{n-1},\om_n\}$.
  \end{enumerate}

  These are precisely the weights listed in the statement of Theorem~\ref{th1}.

  Now for Theorem~\ref{re1}, we let $\om$ be as above, but this time we take $s\in T_{reg}$ and suppose that $s$ is almost cyclic on $V_\om$.
  Then Proposition~\ref{cnn3} applies and we find that either $s$ is strongly regular or $\om$ is as listed in (i) - (iv) above, or $\om = \om_2$, for $(n,p)\ne (3,3)$, or
  $n=4$ and $\om = \om_4$ with $p\ne 3$. Hence we have the result of Theorem~\ref{re1}.\end{proof}

We now consider the group $C_2$.

\bl{c22}
 Let $G=C_2$, 
 and let $s\in T_{reg}$. Then\begin{enumerate}[]

\item{\rm{(1)}} $s$ is cyclic on $V_{\om_1}$ and almost cyclic  on $V_{\om_2}$. If $s$ is not cyclic on $V_{\om_2}$ then
$p\neq 2$ and $-1$ is an  \ei of $s$ on $V_{\om_2}$ of \mult $2$.

\item{\rm{(2)}} There exists $s'\in T_{reg}$, with $s'$ is not strongly regular, and $s'$
 cyclic on $V_{\om_1}$ and on $V_{\om_2}.$

\item{\rm{(3)}} Let $\om\in\Om_p^+$, $p\neq 2$, $\om\not\in\{0,\om_1,\om_2\}$. If $s$ is almost cyclic on $V_\om$, then $s$ is  strongly regular. \end{enumerate}
\el

\begin{proof} The proof of (1) is straightforward and left to the reader.
  For (2), define $s'$ by $\ep_1(s')=a$, $\ep_2(s')=a^3$, where $a\in F^\times$ and $a^{24}\neq 1$.  Then $s'$ is cyclic
  on $V_{\om_1}$ (and hence regular) $s'$ and  is not strongly regular as
  $2\ep_1(s')=(\ep_2-\ep_1)(s')$.
The \eis of $s'$ on $V_{\om_2}$ are $1, a^{\pm2},a^{\pm4}$ (where 1 is to be dropped if $p=2$), so $s'$ is cyclic on $V_{\om_2}$. 

(3) As usual, Theorem~\ref{ag8} implies that $\om$ appears in Table~\ref{tab:omega1} or Table~\ref{tab:omega2}; we have to examine the cases with $p\neq 2$ and $\om\in \{2\om_2,\om_1+\frac{p-3}{2}\om_2,
\frac{p-1}{2}\om_2\}$.

Suppose that $s$ is not strongly regular. As in the proof of Proposition \ref{cnn3}, we can assume that one of the following holds: \begin{enumerate}[(i)]
\item $(2\ep_1)(s)=( 2\ep_2)(s)$,
\item $(2\ep_1)(s)=( \ep_2-\ep_1)(s)$, or
  \item $(2\ep_1)(s) = (-2\ep_1)(s)$. \end{enumerate} By
Lemma \ref{2om}, either $\om\succeq \om_1+\om_2\succ \om_1$ or $\om\succeq 2\om_2$.

Suppose first that $\om\succeq \om_1+\om_2$. Note that
$\pm 2\ep_1\pm\ep_2,\pm \ep_1\pm 2\ep_2,\pm\ep_1,\pm\ep_2\in\Om(V_{\om_1+\om_2})$
If (i) holds, then
$ (2\ep_1-\ep_2)(s) =\ep_2(s)$ and $ (-2\ep_1+\ep_2)(s) =-\ep_2(s)$. As $s$ is almost cyclic  on $V_\om$, it follows that
$\ep_2(s)=-\ep_2(s)$, which is a
contradiction as $2\ep_2$ is a root. If (ii) holds then $(2\ep_1)(s)=( \ep_2-\ep_1)(s)$ implies $ (2\ep_1-\ep_2)(s)=( -\ep_1)(s)$ and
$ (-2\ep_1+\ep_2)(s)=\ep_1(s)$. As  $s$ is almost cyclic on $V_\om$ and $\pm  (2\ep_1-\ep_2)$, $\pm \ep_1\in\Om(V_{\om_1+\om_2})$, it follows that
$ \ep_1(s)= (-\ep_1)(s)$, whence $(2\ep_1)(s)=1$, which contradicts $s$ being regular. Finally, if (iii) holds, so that $(2\ep_1)(s) = (-2\ep_1)(s)$, then $(2\ep_1\pm\ep_2)(s) = (-2\ep_1\pm\ep_2)(s)$ which implies as before that $\ep_2(s) = (-\ep_2)(s)$, giving the usual contradiction.

Suppose  that $\om\succeq 2\om_2$. As $2\om_2\succ 2\om_1\succ \om_2$, we observe that $\pm 2\ep_1\pm2\ep_2$, $\pm 2\ep_1$, $\pm 2\ep_2$, $\pm\ep_1\pm\ep_2$ are
weights of $V_\om$. If (i) holds then $\pm (2\ep_1-2\ep_2)(s)=1$ and hence the \mult of the \ei 1 of $s$ on $V_\om$  is at least 2. In addition,
$(2\ep_1)(s)=( 2\ep_2)(s)$, so $( 2\ep_2)(s)=1 $ as $s$ is almost cyclic  on $V_\om$. This is a contradiction.
If  (ii) holds, that is, $(2\ep_1)(s)=( \ep_2-\ep_1)(s)$ then  $(-2\ep_1)(s)=(- \ep_2+\ep_1)(s)$.
So  $( \ep_2-\ep_1)(s)=-1$ and the \mult of $-1$ as an \ei of $s$ on $V_\om$ is at least 2. As $\pm (2\ep_1-2\ep_2)\in\Omega(V_{2\om_2})$ and
$(2\ep_1-2\ep_2)(s)=1$, we conclude that 1 is an \ei of $s$ on $V_\om$ is of \mult at least 2. This is a contradiction, as $s$ is almost cyclic  on
$V_\om$. Finally, if (iii) holds, then we have $(2\ep_1)(s) = (-2\ep_1)(s)$ , and so $(2\ep_1+2\ep_2)(s) = (-2\ep_1+2\ep_2)(s)$ we deduce that 
$(2\ep_2)(s) = 1$, a contradiction.\end{proof}

\begin{corol}\label{cor:c2}  Theorem~$\ref{th1}$ and Theorem~$\ref{re1}$ hold for $G=C_2$. 

\end{corol}

\begin{proof} Recall that for $s\in T$ and $\om\in \Om_p^+$, $\om\ne 0$, if $s$ is almost cyclic on $V_\om$, then $\om$ appears in
  Table~\ref{tab:omega1} or Table~\ref{tab:omega2}. In particular, for $p=2$, the only weights which need to be considered are $\om_1$ and
  $\om_2$ and both of these are in the statements of the theorems. Henceforth we assume that $p\ne 2$.  If $s\in T$ is cyclic on $V_\om$ for some $\om\in\Om_p^+$, $\om\ne 0$, then again consulting Table~\ref{tab:omega1}, we find that $\om = \om_1$ or $\om=\om_2$, as in the conclusion of
  Theorem~\ref{th1}. So finally, we are left to consider the case where $s\in T_{reg}$, $\om\in\Om_p^+$, and $s$ is almost cyclic on $V_\om$. Then here Lemma~\ref{c22}(3) shows that either $s$ is strongly regular or $\om \in\{0,\om_1,\om_2\}$ as claimed in Theorem~\ref{re1}.\end{proof}

To complete our consideration of the classical type groups, we must establish the existence of elements as claimed in Theorem~\ref{rr4}.

\bl{bn3}
Let $G=B_n$ with $n\geq 3$, or $G=C_n$, with $p=2$ and $n\geq 3$. Then there exists  $s\in T_{reg}$ such that $s$ is not
strongly regular, but cyclic on each of  the modules $V_{\om_1}$ and $ V_{\om_n}$. Moreover, if $G=B_n$, then we may choose $s$ so
that $s$ is cyclic on $V_{\om_1}\oplus V_{\om_n}\oplus \delta_{p,2}V_0$.\end{lemma}

\begin{proof} For $0\leq i\leq n-3$, let  $b_i\in F^\times$ be a
primitive $p_i$-th root of unity where $p_0\ld p_{n-3}$ are distinct primes greater than $11$.
Set  $X=\diag(b_0,(b_0)^3, (b_0)^5,b_1\ld b_{n-3})$.   Now set $t=\diag(X,1,X\up)$, where the 1 is to be
dropped  if $G=C_n$. For a suitable choice of basis for the natural module for $G$, we have $t\in \SO_{2n+1}(F)$, respectively 
$t\in \Sp_{2n}(F)$, for $G=B_n$, respectively $G=C_n$.

We first note that $t$ is cyclic on $V_{\om_1}$. In addition, $(\ep_2-\ep_1)(t)=b_0^2=(\ep_3-\ep_2)(t)$, so $t$ is not strongly regular. 

Consider first the case $G=B_n$ and let $s\in G$ be a semisimple element acting on $V_{\om_1}$ as $t$ does. In particular, $s$ is regular
(since cyclic on $V_{\om_1}$), but  $s$ is not strongly regular.
Now the weights of $V_{\om_1}\oplus V_{\om_n}\oplus \delta_{p,2}V_0$ are $0,\pm\ep_i\ (1\leq i\leq n), \frac{1}{2}(\pm \ep_1\pm\cdots \pm \ep_n)$. Hence, the
remaining claim will follow if we show that the elements $\pm\ep_i(s),1,(\frac{1}{2}(\pm \ep_1\pm\cdots \pm \ep_n)(s)$ are distinct.
As $|t|$ is odd, we may assume that $|s|$ is odd as well. Hence, it suffices to show that
the eigenvalues $2\ep_i(s),1,(\pm \ep_1\pm\cdots \pm \ep_n)(s)$ are distinct. The elements $(\pm \ep_1\pm\cdots \pm \ep_n)(s)$ have the form
$(b_0)^{\pm m}x$, for $m\in\{1,3,7,9\}$ and for some $x\in\{b_1^{\pm1 }\cdots b_{n-3}^{\pm1 }\}$. As $p_i>11$, all these
elements are distinct, and differ from $1,(b_0)^{\pm 2}  , (b_0)^{\pm 6}$, $b_0^{10}$, and $b_1^{\pm2}\ld b_{n-3}^{\pm2}$, whence the result.

The argument for $G=C_ n$ is entirely similar as the weights of $V_{\om_n}$ are $\pm\ep_1\pm\cdots\pm\ep_n$.\end{proof}

\bl{dn3}
Let $G=D_n$, $n\geq 4$. Then there exists  $s\in T_{reg}$ such that $s$ is not almost cyclic on $V_{\om_2}$ (equivalently, $s$ is not strongly regular),
but cyclic on $V_{\om_1}\oplus V_{\om_{n-1}} \oplus V_{\om_n}\oplus V_0$.\end{lemma}

\begin{proof} As $\om_a=\om_2$, by Lemma \ref{sa1}, $s'\in T$ is almost cyclic on $V_{\om_2}$ if and only if $s'$
is strongly regular. We now exhibit the existence of the $s$ as claimed in the statement.

The group  $G=D_n$ is a maximal rank subgroup of  $B_n$ generated by  root subgroups corresponding
to long roots. The element $s$ introduced in Lemma \ref{bn3} is
therefore contained in a subgroup $G=D_n$, and $s$ is regular as an element of $D_n$, since it is regular in $B_n$.
Moreover, we have seen in the proof of Lemma \ref{bn3} that $(\ep_2-\ep_1)(s)=(\ep_3-\ep_2)(s)$;  as $\ep_2-\ep_1$ and
$\ep_3-\ep_2 $ are  long roots of $B_n$, they are roots of the subgroup $D_n$. So $s$ is not strongly regular in $D_n$ and by
Lemma~\ref{sa1}, it follows that $s$
is not almost cyclic on $V_{\om_2}$.
Furthermore,   for the purposes of this proof, let $\eta_1,\dots,\eta_n$ denote the fundamental dominant weights of $B_n$. Then,
the \ir module for $B_n$ with highest weight $\eta_n$ restricts to $D_n$
as $V_{\om_{n-1}}\oplus V_{\om_n}$.  In addition, $V_{\eta_1}|_H = V_{\omega_1} \oplus (1-\delta_{p,2})V_0$.
Now Lemma~\ref{bn3} shows that $s$ is cyclic on $V_{\eta_1}+V_{\eta_n}+\delta_{p,2}V_0$, and so $s$ is cyclic on
$V_{\omega_1}+V_{\omega_{n-1}}+V_{\omega_n}+V_0$, as claimed. \end{proof}

\begin{rem}\label{ex:bdc}{\rm Note that the above two results provide a proof of Theorem~\ref{rr4}(1) for $G=B_n$, $n\geq 3$, $G = D_n$ with $n\geq 4$, and
    $G = C_n$ for $n\geq 3$ and $p=2$.}\end{rem}

\def\rss{regular semisimple }

\bl{sr3} Let $G=C_n $, $n\geq 2$. Then  there exists $s\in T_{reg}$, not strongly regular, with $s$ almost cyclic  on $V_{\om_2}$.\el

\bp Define $s$ by $\ep_1(s)=a$,  $\ep_2(s)=a^3$,  $\ep_i(s)=b_i$, for $a,b_1\ld b_{n-2}\in F^\times$. Then $s$ is not
strongly regular as $(2\ep_1)(s)=a^2=(\ep_2-\ep_1)(s)$, and $(2\ep_1)(s)$, $\ep_2-\ep_1$ are roots of  $G$. Moreover, one may choose the $a,b_i$ such that $s$ is
regular. 
The \eis of $s$ on $V_{\om_2}$ are
$1,a^{\pm 2},a^{\pm 4},a^{\pm 1}b_i^{\pm 1},a^{\pm 3}b_i^{\pm 1},b_i^{\pm 1}b_j^{\pm 1}$, $i\ne j$.   
We may choose $a,b_1\ld b_{n-2}\in F^\times$ such that the  eigenvalues different from $1$ are distinct. Indeed, the non-zero weights of $V_{\om_2}$ are roots and by choosing $s$ regular, we know that the multiplicity of the eigenvalue $1$ is
precisely the multiplicity of the zero weight in $V_{\om_2}$, while all other weights have multiplicity 1.  Hence $s$ is almost
cyclic on $V_{\om_2}$. \end{proof}

\begin{rem}\label{b23}{\rm{A similar argument works for $G=B_n$, with $p=2$; for this case, set $\ep_1(s) = a$ and $\ep_2(s) = a^2$ and the remaining values as for $C_n$.}}\end{rem}

\bl{om4} Let $G=C_n$, $n\geq 2$. Then there exists $s\in T_{reg}$, $s$ not strongly regular, which separates the weights of $V_{\om_k}$,
for $k=1\ld n$.
\el

\bp Let $a,b_i\in F^\times$, for $1\leq i\leq n$, and define $s\in T$ by $\ep_i(s) = b_i$ for $1\leq i\leq n$  and in addition $b_1 = a$ and
$b_2 = a^3$.  We first note that $s$ is not strongly regular as $2\ep_1(s) = (\ep_2-\ep_1)(s)$. We now show that by choosing appropriate values
for the
$b_i$, the element $s$ separates the weights of $V_{\om_k}$, as claimed. This will also show that $s$ is regular as $s$ is then cyclic on $V_{\om_1}$.

The non-zero weights of $V_{\om_k}$ are $\pm\ep_{i_1}\pm...\pm \ep_{i_j}$ for $j\leq k$ with $k-j$ even and $1\leq i_1<...< i_j\leq n$.
This follows from the fact that $V_{\om_j}$ occurs as a composition factor of $\wedge^j(V_{\om_1})$, for all $1\leq j\leq n$, and the subdominant weights of $\om_k$,
for $k\geq 2$, are of the
form $\om_{k-2i}$ for $0\leq i\leq \lfloor \frac{k}{2}\rfloor$. 
A weight $\pm\ep_{i_1}\pm...\pm \ep_{i_j}$ takes the value  $b_{i_1}^{\pm 1}\cdots b_{i_j}^{\pm 1}=xy$,  on $s$, where $x = a^{\pm m}$ for some
$0\leq m\leq 4$ and $y$ is
a product of certain $b_i^{\pm 1}$, with $i>2$. (In particular, if $n=2$ then the weights of $V_{\om_k}$ take values among
$a^{\pm 1},a^{\pm 3},a^{\pm 2},a^{\pm 4},1$ on $s$.)
For a suitable choice of $a$ and $b_i$, these values are distinct.  For example, we can choose $b_i$ to be a primitive $p_i$-th root of unity for
$3\leq i\leq n$  for
primes $p_3>...>p_{n}$,  and $a$ to be a primitive $p_0$-th root of unity for a prime $p_0>{\rm max}\{p_3, 7\}$. Therefore, $s$ separates the
weights of $V_{\om_k}$ for every fixed $k$.\end{proof}

\begin{propo}\label{c43}  Let $G=C_n, n\geq2$, $p\ne 2$. \begin{enumerate}[]
  \item{\rm{(1)}} Let $\om\in\Om_p^+$ be as follows:\begin{enumerate}[]
  \item{\rm {(i)}} $\om=\om_1$,
  \item{\rm{(ii)}} $\om\in\{\om_{n-1},\om_n\}$, $n\geq 3$, $p=3$, or
  \item{\rm{(iii)}} $\om = \om_3$, $n=3$, $p\ne 3$. \end{enumerate}
  Then there exists  $s\in T_{reg}$, with $s$ not
  strongly regular but cyclic on $V_{\om}$.\medbreak
\item{\rm{(2)}} Let $\om\in\Om_p^+$ be one of:\begin{enumerate}[]
  \item{\rm{(i)}} $\om = \om_4$, $n=4$ and $p\ne 3$, or
  \item{\rm {(ii)}} $\om = \om_2$ with $(n,p)\ne(3,3)$.
  \end{enumerate}
 Then there exists  $s\in T_{reg}$, with $s$ not strongly regular but almost cyclic on $V_{\om}$.
  \end{enumerate}

 \end{propo}

\begin{proof} This follows from Lemma \ref{om4} and Tables~\ref{tab:omega1} and \ref{tab:omega2}. Indeed, if $s$ separates the weights of $V$
then $s$ is cyclic on $V$ whenever all weights of $V$ are of \mult 1, and almost cyclic if exactly one weight
is of \mult greater than 1.\enp

\bl{th4} Theorem {\rm \ref{rr4}} is true if G is of classical Lie type.\el

\bp  For $G$ of type $A_n$ this was established in Section~\ref{sec:An}. For Theorem~\ref{rr4}(1), we refer to Remark~\ref{ex:bdc} and note that it remains to consider $G=C_2$ and $G=C_n$ with
$n\geq 3$ and $p\ne 2$. For $G = C_2$, we must establish the existence of $s\in T$, cyclic on $V_{\om_1}$ and on $V_{\om_2}$,
with $s$ not strongly regular. This follows from Lemma~\ref{c22}(2).

For $G = C_n$, $n\geq 3$ and $p\ne 2$, we must establish the existence of $s\in T$, $s$ not strongly regular, with $s$
cyclic on \begin{enumerate}[(a)]
\item $V_{\om_1}$,
\item $V_{\om_3}$ when $n=3$,  and 

\item $V_{\om_{n-1}}$ or $V_{\om_n}$, when $p=3$.
  \end{enumerate}

  (If necessary, we may have different $s$ for each of the modules.) This is precisely the statement of Proposition~\ref{c43}(1).

  Now we turn to the list of modules to be treated for Theorem~\ref{rr4}(2). For $n\geq 3$, $\om=\om_2$, and $G = C_n$, or $G = B_n$ with $p=2$, the result follows from Lemma~\ref{sr3} and Remark~\ref{b23}. For $G = C_4$ with $\om=\om_4$, $p\ne 2,3$, the result follows from Proposition~\ref{c43}(2).\end{proof}

\section{ Groups of exceptional types and Theorem~\ref{td4}}

In this section we will establish all of the main results in case $G$ is of exceptional type and deduce Theorem~\ref{td4}.

\bl{e1e} Theorem {\rm \ref{re1}} holds for  $G$ of exceptional type.\el

\bp Let $\om\in\Om_p^+$ and $s\in T$ be almost cyclic on $V_\om$. Then Theorem~\ref{ag8} and Tables~\ref{tab:omega1} and
\ref{tab:omega2} imply that either $\omega = \omega_a$, or $\omega$ is the highest short root (in case $G=G_2$ or $F_4$), or $\omega$ is
a miniscule weight. The miniscule weights appear in the conclusion of Theorem~\ref{re1}, as well as the highest short roots. Moreover,
Lemma~\ref{sa1} shows that if $s$ is almost cyclic on $V_a$, then either $s$ is strongly regular or $(G,p) = (G_2,3)$ or $(F_4,2)$. The
latter configurations are also in the list of exceptions in the statement of Theorem~\ref{re1}, hence the result.\end{proof}

\bl{e2e} Theorem {\rm \ref{th1}} holds for  $G$  of exceptional type.\el

\bp Let $\om\in\Om_p^+$  and $s\in T$ be cyclic on $V_\om$. Then all weights of $V_\om$ have multiplicity $1$ and now comparing the statement of
Theorem~\ref{th1} with Table~\ref{tab:omega1} gives the result. \enp

We are left with proving Theorem \ref{rr4}; we treat each of the groups in turn.


\begin{propo}\label{df4}
 Let $G$ be of type $F_4$. There exists   $s\in T_{reg}$ such that $s$ is almost cyclic  on $V_{\omega_4}$ but not strongly regular.
  If $p=3$ then $s$ can be chosen to be cyclic on  $V_{\omega_4}$. 
\end{propo}

\begin{proof} Let $V = V_{\om_4}$. Recall that $G$ contains a subsystem subgroup $H$ of type $B_4$, generated by the long root subgroups of $G$, so  $T\leq H$.
For the purposes of this proof, let $\eta_i$, $1\leq i\leq 4$ denote the fundamental dominant weights of $H=B_4$ (with respect to an appropriate choice of
  Borel subgroup of $H$).
Then the non-zero  weights of $V|_H$
are exactly the non-zero weights  of the $FH$-module $V_{\eta_1}+V_{\eta_4}$. By comparison of the dimensions we conclude that the non-trivial
composition factors of  $V|_H$  are $V_{\eta_1}$ and $ V_{\eta_4}$, and  there are  $(1-\delta_{p,3})(1+\delta_{p,2})$ trivial composition factors.

Now let $s\in H$ be as in Lemma \ref{bn3}; then,  $\alpha(s) = \beta(s)$, for $\alpha$ and $\beta$ distinct roots of $H$, so
$s$ is not strongly regular in $F_4$. Moreover,  $s$ is
cyclic  on $V_{\eta_1}+V_{\eta_4}$ and so almost cyclic
on $V$.   Then  Theorem~\ref{c99} implies that $s$ is regular in $G$. If $p=3$, then $V|_H = V_{\eta_1}+V_{\eta_4}$ and Lemma \ref{bn3} shows that $s$ is cyclic on $V$.
\end{proof}
\begin{rem}\label{nr2}{\rm Let $G=F_4$, $p=2$. \begin{enumerate}
  \item There exists an element $s\in T_{reg}$, $s$ not strongly regular, which is almost
  cyclic  on $V_a=V_{\om_1}$ (and not almost cyclic on $V_{\om_4}$). Recall that,
  identifying $T$ with $(F^\times)^4$ and $\ep_i$ the $i$-th coordinate function for $1\leq i\leq 4$, the roots of $F_4$ are
  $\pm\ep_i, \frac{1}{2}(\pm\ep_1\pm\ep_2\pm\ep_3\pm\ep_4)$ (short) and
  $\pm\ep_i\pm\ep_j$, $1\leq i<j\leq 4$ (long).
  Let $a_i\in F^\times$ be a primitive $p_i$-root of unity, where $p_1=3,p_2=5,p_3=7$. Define $s\in T$ by $\ep_1(s)=a_1,$ $\ep_2(s)=a_2,$
  $\ep_3(s)=a_1a_2,$ $\ep_4(s)=a_3$.
The non-zero weights  of $V_a$ are $\pm \ep_i\pm\ep_j$ $(1\leq i<j\leq 4)$ so $s$ is almost cyclic on $V_a$. Then
$(\ep_ 1+\ep_2-\ep_3+\ep_4)(s)=(-\ep_ 1-\ep_2+\ep_3+\ep_4)(s)=\ep_4(s)=a_3$ so $s$ is not almost cyclic  on $V_{\om_4}$. Note that $s$ is regular as
$\al(s)\neq 1$ for every root $\al$, but it is not strongly regular.
This justifies one of the claims of Remark \ref{rr3}, and together with Proposition~\ref{df4}, completes the proof of
Theorem~\ref{rr4} for the group $G = F_4$.


  \item There exists a non-regular element $s\in T$ such that $V^s_a=V_a^T$. Indeed, let $\ep_i(s)=t_i\in F^\times$, where $t_1=1$. Then $s$ is not regular. The
  weights of $V_a=V_{\om_1}$ are $0,\pm\ep_i\pm\ep_j$, so the \eis of $s$ on the non-zero weights of $V_a$ are $1,(\pm\ep_i\pm\ep_j)(s)$, that is,
  $1,t_i^{\pm 1},t_i^{\pm 1}t_j^{\pm 1}$,
 where $ i,j=2,3,4, i\neq j$. For an appropriate choice of $t_i\in F^\times$, these \eis differ from 1,
 and the claim follows. This proves a claim in Remark~\ref{rr3} for the group $G = F_4$.\end{enumerate}}\end{rem}

\begin{propo}\label{th1_e6}
  Let $G=E_6$. Then there exists $s\in T$, not strongly regular, which is cyclic on $V_{\om_1}$ and on $V_{\om_6}$.
\end{propo}

\begin{proof} It suffices to consider $V=V_{\om_1}$. Let $H<G$ be a subsystem subgroup of type $D_5$.
  For the purposes of this proof, let $\eta_i$, for $1\leq i\leq 5$, denote the fundamental dominant weights of $H$ (with respect to an appropriate choice of
  maximal torus $T_H$ and Borel subgroup of $H$).
  Then, we have that $V|_H = V_{\eta_1}\oplus V_{\eta_4}\oplus V_0$. (Here we have made a choice of the labeling of the roots of $H$.)

  Now let $s\in H$ be as in Lemma~\ref{dn3}. Then $s$ is cyclic on $V_{\eta_1}+V_{\eta_4}+V_{\eta_5}+V_0$, so in particular on $V$.
  As well, $s$ is not strongly regular in $H$ and since roots of $H$ are roots of $G$, $s$ is not strongly regular in $G$. \end{proof}

\begin{propo}\label{e66}
  Let $G$ be of type $E_6$. Then there exists  $s\in T_{reg}$ such that $s$ is not strongly regular but almost cyclic  on $V_{\omega_1}$ and on
  $V_{\omega_6}$. 
\end{propo}

\begin{proof} It suffices to establish the result for $V_{\om_1}$. Let $H\leq G$ be of type $F_4$. Let $s$ be as constructed in the proof of
  Proposition~\ref{df4}. For the purposes of this proof, let $\eta_i$, $1\leq i\leq 4$ denote the fundamental dominant weights of $H$
  (with respect to a fixed choice of Borel subgroup and maximal torus). Then $V_{\om_1}|_H = V_{\eta_4}\oplus(1+\delta_{p,3})V_0$. Since $s$ is almost
  cyclic on $V_{\eta_4}$, and cyclic if $p=3$,  and the zero weight has multiplicity 2 in $V_{\eta_4}$ if $p\ne 3$, we see that $s$ is almost cyclic on $V_{\om_1}$.
  We claim that $s$ is regular but not strongly regular. Recall that $V_a = V_{\omega_2}$. We note that

  $$V_a|_H = \begin{cases}V_{\eta_1}\oplus V_{\eta_4},& \mbox{if } p\ne 2,3;\cr
        V_{\eta_1}\oplus V_{\eta_4}\oplus V_0,&\mbox{if } p=3;\cr
        V_{\eta_1}\oplus V_{\eta_4}\oplus V_{\eta_4},&\mbox{if } p=2.\end{cases}$$

      To see that $s$ is regular, we need to show that the multiplicity of the eigenvalue $1$ for $s$ on $V_a$ is equal to $\dim T = 6$.
      Note that $s$ is regular in $H$ and so the eigenvalue $1$ on $V_{\eta_1}$ has multiplicity $4-2\delta_{p,2}$.
      Also, the non-zero weights of $V_{\eta_4}$ are roots of $H$ and the zero weight has multiplicity $2-\delta_{p,3}$ in $V_{\eta_4}$; so again the eigenvalue 1 has multiplicity $2-\delta_{p,3}$ in $V_{\eta_4}$. In each of the above decompositions, we find that the eigenvalue 1 for $s$ on $V_a$ has multiplicity 6.

      Finally, to see that $s$ is not strongly regular, we appeal to Lemma~\ref{sa1}. 
      It suffices to see that $s$ is not almost cyclic on $V_a$. Recall that $\eta_4\prec\eta_1$. Using Theorem~\ref{premet} and the above decompositions of $V_a|_H$, we find  that all non-zero weights of
      $V_{\eta_4}$ occur with multiplicity 2 in $V_a$, whence the result.\end{proof}

We will use the following result when treating the group $G=E_7$.

\begin{propo}\label{aa7}
  Let $G=A_n$, $n\geq 4$. Then there exists $s\in T_{reg}$, $s$ not strongly regular,  such that $s$ is not almost cyclic on  $V_{\omega_1+\omega_n}$ and cyclic
  on $V=V_{\omega_2}\oplus V_{\omega_{n-1}}$.
   \end{propo}

   \begin{proof}  We start by considering a certain subvariety of $F^m$. Let $m\geq 4$. Set
     $X = \{(x_1,\dots,x_{m})\in F^m \ |\ x_1\cdots x_m=1\}$. Then $X$ is isomorphic to the principal open subset $D(f)\subset F^{m-1}$
     defined by the function $f\in F[T_1,\dots,T_{m-1}]$, $f = T_1\cdots T_{m-1}$. Indeed, the associated
     $F$-algebra of $X$ is isomorphic to $F[T_1,\dots,T_{m-1}, (T_1\cdots T_{m-1})^{-1}]$ which is (isomorphic to) the localization of
     $F[T_1,\dots,T_{m-1}]$ with respect to the multiplicative set $S = \{(T_1\cdots T_{m-1})^\ell\ |\ \ell\geq 0\}$. In particular,
     $X$ is irreducible of dimension $m-1$.
 Now consider the closed subset $Y$ of $X$, $Y = \{(x_1,\dots,x_m)\in X\ |\ x_1^2 = x_2x_3\}$. Then the associated affine $F$-algebra is
 $F[X]/(T_1^2-T_2T_3)$. Since $F[X]$ is a localization of $F[T_1,\dots,T_{m-1}]$ and $T_1^2-T_2T_3$ is an irreducible polynomial
 function in $F[T_1,\dots T_{m-1}]$, $Y$ is also an irreducible subvariety of $F^m$, of dimension $m-2$.

 We now apply the above reasoning to a maximal torus of $G={\rm SL}_{n+1}(F)$, namely the usual torus of diagonal matrices of
 determinant 1. Set $s = {\rm diag}(d_1,\dots,d_{n+1})\in T$, where $d_i\in F^\times$ and
 $\prod_{i=1}^{n+1} d_i = 1$. If we assume further that $d_1^2 = d_2d_3$ then $(\ep_1-\ep_2)(s) = (\ep_3-\ep_1)(s)$ and so $s$ is not
 almost cyclic on $V_{\omega_1+\omega_n}$.
   Let $Y = \{s\in T\ |\ d_1^2=d_2d_3\}$ so that by the
   discussion of the first paragraph, $Y$ is an irreducible closed subvariety of $T$ of dimension $n-1$.
   The module $V$ has $1$-dimensional $T$-weight spaces. For weights $\lambda,\mu\in\Om(V)$, with $\lambda\ne\mu$, let $K_{\lambda\mu} = \ker(\lambda-\mu)$. The set
   $K:=\cup_{\lambda,\mu}K_{\lambda\mu}$ is a proper closed subset of $T$. Since $Y$ is irreducible, if $Y\subseteq K$,
   then there exist $\lambda,\mu\in\Omega(V)$, with $\lambda\ne\mu$, such that $Y\subseteq \ker(\lambda-\mu)^\circ$ and dimension considerations show that
   these two sets must be equal. But the weights of $V$ are of the form $\pm(\ep_i+\ep_j)$, $1\leq i<j\leq n+1$. So the possible weight
   differences
   $\lambda-\mu$ (up to a sign) are $\ep_i+\ep_j\pm(\ep_k+\ep_\ell)$, $\ep_j-\ep_\ell$, $2\ep_i+\ep_j+\ep_\ell$ and $2(\ep_i+\ep_j)$, for
   $i,j,k,\ell$ distinct. So there are no distinct weights $\lambda,\mu$ with $Y=\ker(\lambda-\mu)^\circ$. Hence $Y$ does not lie in $K$ and the elements in
   $Y\setminus K$ are almost cyclic, indeed cyclic, on $V$ and not almost cyclic on $V_{\omega_1+\omega_n}$. So these elements are regular by Lemma \ref{ss2}
   and not strongly regular by Lemma~\ref{sa1}.\end{proof}

 We now use the preceding result to establish the following

     \begin{propo}\label{ea7}
      Let $G=E_7$. There exists an element $s\in T_{reg}$, with $s$ not strongly regular, and such that $s$ is cyclic
       on $V_{\omega_7}$.
        \end{propo}

        \begin{proof} Set $V = V_{\om_7}$. Let $H\leq G$ be a maximal rank subgroup of type $A_7$ and for the purposes of this proof,
          let $\eta_i,1\leq i\leq 7$, denote the fundamental dominant weights of $H$ (with respect to an appropriate choice of maximal torus $T_H$ and Borel
          subgroup). Then $V|_H = V_{\eta_2}+V_{\eta_6}$. 
          Let $s\in T_H$,  be as in Proposition~\ref{aa7}, so that $s$ is cyclic on $V$ and not strongly regular in $H$. Then $s$ is not strongly regular
          in $G$ as the roots of $H$ are roots of $G$. Finally, Lemma~\ref{ss2} shows that $s$ is regular.\end{proof}

      For the proof of Theorem~\ref{rr4}, it remains to treat the group $G=G_2$. We investigate in detail the two modules corresponding to the fundamental dominant
      weights.

\begin{propo}\label{2g3} Let $G=G_2$, $V_i=V_{\om_i}$ for $i=1,2$.\begin{enumerate}[]

  \item{\rm {(1)}} Let $s\in T_{reg}$ and suppose that $s$ is not cyclic on $V_1$. Then $p\neq 2$  and the \eis of $s$   on
    $V_1$ are $\{1,-1,-1,-b,b\up, b,-b\up\}$,  where $b\in F^\times$, $b^2\neq 1$. In particular, $s$ is  almost cyclic  on $V_1$ \ii $b^2\neq  -1$.
 \item{\rm{(2)}} Let $s\in T_{reg}$ and assume that $p=3$.  Then  $s$ is almost cyclic  on $V_2$ and cyclic   on $V_1$. More precisely, if
$s$ is not cyclic on $V_2$ then, for some $a\in F^\times, a^{4}\neq 1$, the \eis of $s$   on $V_2$ are
$\{1,-1,-1, a^{\pm3},-a^{\pm3}\}$,  and those 
on $V_1$ are
$\{1,a^{\pm1}, -a^{\pm1},$ $-(a^{\pm2})\}$.

 \item{\rm {(3)}} There exists  $s\in T_{reg}$, not strongly regular, with $s$ cyclic on $V_1$.
 \item{\rm {(4)}} Let $p=3$. Then there exists  $s\in T_{reg}$, not strongly regular, with $s$ cyclic on $V_2$.
 \item{\rm{(5)}} Let $p=3$. Then there exists a non-regular element $s\in T$,  with $V_a^s = V_a^T$.\end{enumerate}
\end{propo}

\begin{proof} Set $\al_1(s)=a$, $\al_2(s)=b$, for some $a,b\in F^\times$.

(1) Here we suppose $s$ to be regular and not cyclic on $V_1$. The weights of $V_1$ are the short roots, and the weight 0 if $p\ne 2$, all with multiplicity 1.
Then $(\al_1+\al_2)(s)=ab$,
$(2\al_1+\al_2)(s)=a^2b$ so the eigenvalues of $s$ on $V_1$ are $1$ (if $p\ne 2$), $a^{\pm1},(ab)^{\pm1},(a^2b)^{\pm1}$. Since $s$ is not cyclic on $V_1$, at least
two of the eigenvalues are equal. As $s$ is regular, $\al(s)\neq 1$ for every root $\al$, in particular,
$1\notin \{a,ab,a^2b\}$.
As $W(G)$ is transitive on the short roots, we can assume that $a\in\{a\up, (ab)^{\pm1},(a^2b)^{\pm1}\}$. Now if 
$a\in \{(ab)^{\pm1},(a^2b)^{\pm1}\}$, then $1\in \{b,a^2b,ab,a^3b\}$. But $\al_2,2\al_1+\al_2$, $\al_1+\al_2,3\al_1+\al_2$ are roots, so
this contradicts  $s$ being regular.
So we have $a=a \up$. Then  $a^{2}=1$ and $a=\al_1(s)\neq 1$, so  $p\neq2 $ and $a=-1$. Then the \eis of $s$ on $V_1$ are
$1,-1,-1, -b,-b\up, b, b\up$, where $b\neq \pm 1$ as $s$ is regular, establishing the first claim. Finally, 
$s$ is almost cyclic  on $V_1$ if and only if $b^2\ne-1$.   

 (2) Here we suppose $s$ to be regular and   $p=3$.  The weights of $V_2$ are the long roots and the weight 0,
 so the \eis of $s$ on $V_2$ are $1,b^{\pm 1}, (a^3b)^{\pm 1},(a^3b^2)^{\pm 1}$. Note that $s$ regular implies that $b\ne 1$.
 As above, if $s$ is not cyclic on $V_2$,  we can assume that
$\al_2(s)\in \{-\al_2(s), \pm(3\al_1+\al_2)(s),\pm(3\al_1+2\al_2)(s)\}$. As $s$ is regular,  $\al_2(s)\neq (-3\al_1-\al_2)(s)$ and
$\al_2(s)\neq (3\al_1+2\al_2)(s)$.
If $\al_2(s)=(3\al_1+\al_2)(s)$ then $(3\al_1)(s) =1$, so $a^3=1$, and if  $\al_2(s)=(-3\al_1-2\al_2)(s)$, then $(ab)^3=1$. But since
$p=3$, this implies that $a=1$, respectively $ab=1$, contradicting that $s$ is regular. So these two cases are ruled out, and we are
left with the case where $\al_2(s)=-\al_2(s)$.  So $b^2=1$, and hence $b=-1$, and $a^2\ne 1$ since $s$ is regular.

We now have that the \eis of $s$ on $V_2$ are
$\{1,-1,-1, a^3,-a^3, a^{-3},-a^{-3}\}$. Suppose for a contradiction that $s$ is not almost cyclic on $V_2$; then either $a^3 = -1$ or $a^6 = -1$. In the first case,
$a = -1$, and hence
$ab=1$ contradicting $s$ regular. Thus $a^6=-1$,
equivalently, $a^2=-1$. But then, $(2\al_1+\al_2)(s)=a^2b=1$, again contradicting that $s$ is regular. Hence $s$ is almost cyclic on $V_2$ as claimed. Note that $a^2\ne -1$, else $(2\alpha_1+\alpha_2)(s) = 1$. 
   Now the \eis of $s$ on $V_1$ are $ \{1,a,a\up, -a,-a\up, -a^{2},-a^{-2}\}$, and hence $s$ is cyclic on $V_{1}$.

   (3)  Let  $a=b$, so that the \eis of $s$ on $V_1$ are $\{a^{\pm 3},a^{\pm 2},a^{\pm 1}$, $1\}$, where 1 is to be dropped if $p=2$;
   they are distinct if $a^k\neq 1$ for $k\leq 6$; for instance choose $a$ to be a primitive $7$-th root of unity. Then $s$ is cyclic
   on $V_{1}$ (hence regular) and is not strongly regular   as $\al_1(s)=\alpha_2(s)=a$.

   (4) Take again $a=b$, so  that the eigenvalues of $s$ on $V_2$ are $a^{\pm1}, a^{\pm4}, a^{\pm 5}, 1$. Then choosing $a$ appropriately,
   we see that $s$ is cyclic on $V_2$, regular and not strongly regular, as in (3).

   The final statement (5) is straightforward, as we may take $a=1$ so that $s$ is not regular, while the non-zero weights of $V_a$ are
   $\pm(3\alpha_1+2\alpha_2)$, $\pm(3\alpha_1+\alpha_2)$ and $\pm\alpha_2$; so that choosing $b$ such that $b^2\ne 1$,
   we have $V_a^s = V_a^T$.\end{proof}

 We have now completed the proof of Theorem~\ref{rr4} for the exceptional groups. Finally, we conclude the article by proving Theorem~\ref{td4}:

\begin{proof}[Proof of Theorem~\ref{td4}] By Lemma \ref{td2},
$\phi_i(s)$ is cyclic for every $i$. Therefore, by Theorem~\ref{th1}, $s$ is strongly regular unless, possibly, for every $i$, $(G,\lam_i)$ is as in cases
$(2)-(10)$ there. Also, Theorem~\ref{ag8} implies that all non-zero weights of $V_\om$ have multiplicity one, and we then apply  \cite[Theorem 2(2)]{TZ2} and \cite[Proposition 2]{SZ1} to  justify the conditions given in (i) - (iii) of the statement.\end{proof}

\section{Tables}
Tables~\ref{tab:omega1} and \ref{tab:omega2} are taken from \cite[Tables 1,2]{TZ2}. 
We recall here our convention for reading the tables in case ${\rm char}(F)=0$: 
for a natural number $a$ the expressions $p>a$,
$p\geq a$ or $p\ne a$ are to be interpreted as the absence of any restriction, that is, $a$ is allowed to be any natural number.
Note further that when a weight $\lambda$ has coefficients expressed in terms of $p$, we are assuming that ${\rm char}(F) = p>0$.

\begin{table}[h]
$$\begin{array}{|l|c|}\hline
~~~~~~~~~~G&  \\
\hline
\hline
A_1&a\om_1, 1\leq a<p\\
\hline
A_n, n>1&a\om_1, b\om_n, \ 1\leq a,b<p\\
& \om_i,\  1<i< n\\
&c\om_i+(p-1-c)\om_{i+1}, \  1\leq i<n,\  0\leq c<p \\
\hline
B_n, n>2 &\om_1,\ \om_n \\
\hline
C_n, n>1, p=2&\om_1,\ \om_n\\
\hline 
C_2,p>2&\om_1,\ \om_2,\ 
\om_1+\frac{p-3}{2}\om_2,\ \frac{p-1}{2}\om_2\\
 \hline
C_3&\omega_3\\
\hline
C_n,n>2,p>2&\om_1,\ \om_{n-1}+\frac{p-3}{2}\om_n,\ \frac{p-1}{2}\om_n\\
 \hline
D_n, n>3&\om_1,\ \om_{n-1},\  \om_n\\
 \hline
E_6&\om_1, \ \om_6\\
\hline
E_7&\om_7  \\
\hline
F_4, p=3&\om_4 \\
\hline
G_2, p\neq 3&\om_1\\
\hline 
G_2, p= 3&\om_1,\ \om_2\\
\hline
\end{array}$$
\caption{Non-trivial irreducible $p$-restricted $G$-modules with all weights of
multiplicity 1}\label{tab:omega1}
\end{table}

\newpage${}$

\begin{table}[H]
$$\begin{array}{|l|l|c|c|}
\hline
G& \mbox {conditions}& &\mbox{weight } 0 \mbox{ \mult}\\
\hline
\hline
A_n,&n>1,(n,p)\neq (2,3)&\om_1+\om_n& \begin{cases} n-1&\mbox{ if }p|(n+1)\\ n&\mbox{ if } p\!\not|(n+1)\end{cases}\\
&&& \\
A_3& p>3&2\om_2& 2     \\

 \hline
B_n &n>2,p\ne 2& \om_2&  n \\
&n>2,p= 2& \om_2&  n-{\rm gcd}(2,n) \\
&&&\\
&-&2\om_1&\begin{cases}n&\mbox{ if } p|(2n+1)\\  n+1&\mbox{ if } p\!\not|(2n+1)\end{cases}   \\
\hline
C_n&n>1&2\om_1 &n\\
& n>2, (n,p)\ne (3,3)&\om_2& \begin{cases}n-2& \mbox{ if } p|n\\  n-1& \mbox{ if } p\!\not|n\end{cases}\\
C_2&p\ne5&2\om_2&2\\
C_4&p\ne 2,3&\omega_4&2\\
\hline
D_n&n>3&2\om_1&  \begin{cases}n-2 &\mbox{ if } p|n\\  n-1& \mbox{ if }
p\!\not|\,n\end{cases}\\
 &n>3, p\ne 2& \om_2& n\\
 D_n&n>3, p=2& \om_2&n-\gcd(2,n) \\
\hline
E_6&&\om_2&\begin{cases}5 &\mbox{ if } p=3\\  6& \mbox{ if } p\neq 3\end{cases} \\
\hline 
E_7&&\om_1 &\begin{cases}6 &\mbox{ if } p=2\\  7& \mbox{ if } p\ne2\end{cases}\\
 \hline
E_8&&\om_8 & 8\\
\hline
F_4&&\om_1 &\begin{cases}2& \mbox{ if } p=2\\  4 &\mbox{ if } p\ne2\end{cases}   \\
 & p\neq 3&\om_4 &2 \\
\hline
G_2& p\neq 3&\om_2&2\\
\hline
\end{array}$$
\caption{Irreducible $p$-restricted $G$-modules with non-zero
weights of multiplicity 1 and whose zero weight has \mult 
greater than $1$.}\label{tab:omega2}
\end{table}

\begin{table}[H]
$$\begin{array}{|l|l|l|l|l|l|l|}
    \hline
    &&&&&&\\
    A_1&A_n &B_n &C_2  & C_n& C_3&C_n  \\
    p\ne 2 &n\geq 2&n\geq 3&&n\geq 3, p=2&p\ne 2,3&n\geq 3, p=3\\
    &&&&&&\\
\hline
    \om_1&\om_1\ld \om_n  & \om_1,\om_n  &\om_1,\om_n   & \om_1,\om _n &\om_1,\om_3&\om_1,\om_{n-1},\om_n     \\
    \hline\hline
    &&&&&&\\
   D_n, n\geq 4&  E_6  &E_7  & F_4,p= 3  &G_2, p\neq 3 & G_2, p=3& \\
    &&&&&&\\
    \hline

   \om_1,\om_{n-1},\om_n&\om_1,\om_6& \om_7&  \om_4  &   \om_1& \om_1,\om_2& \\
                                                 \hline
  \end{array}$$
  \caption{Highest weights of the non-trivial $p$-restricted modules in cases $(2)-(10)$ of Theorem~\ref{th1}}\label{tab:results}

\end{table}

\end{document}